\documentclass[reqno,12pt,twosided,a4paper]{amsart}
\usepackage{amssymb,amscd,stmaryrd}
\usepackage{amsmath}
\usepackage{amsfonts}
\usepackage[top=3cm,bottom=3cm,left=2.5cm,right=2.5cm,footskip=17mm]{geometry}
\usepackage[mathcal]{eucal}
\usepackage{array,float}
\usepackage{enumitem}
\usepackage[OT2,T1]{fontenc}
\usepackage[russian,USenglish]{babel}
\usepackage{chngcntr}
\usepackage{apptools}
\AtAppendix{\counterwithin{theorem}{section}}
\usepackage{float}
\usepackage{pifont}
\newfloat{Diagram}{htbp}{dia}
\usepackage{amsmath}
\usepackage{amssymb, xy}
\input xy
\xyoption{all}
\usepackage{pdflscape}
\usepackage{hyperref}
\usepackage[draft]{graphicx}
\usepackage{tikz,pgf}
\usepackage{tikzit}

\tikzstyle{1function}=[fill=white, draw=black, shape=rectangle, minimum width=0.75cm, minimum height=1 cm]
\tikzstyle{2function}=[fill=white, draw=black, shape=rectangle, minimum width=1cm, minimum height=1 cm]
\tikzstyle{3function}=[fill=white, draw=black, shape=rectangle, minimum width=1.5cm, minimum height=1cm]
\tikzstyle{multi function}=[fill=white, draw=black, shape=rectangle, minimum width=5cm, minimum height=1cm]
\tikzstyle{multi function small}=[fill=white, draw=black, shape=rectangle, minimum width=4cm, minimum height=1cm]
\tikzstyle{Multi function smaller}=[fill=white, draw=black, shape=rectangle, minimum width=3.5 cm, minimum height=1 cm]
\tikzstyle{graph node}=[fill=black, draw=black, shape=circle]
\tikzstyle{smallbox}=[fill=white, draw=black, shape=rectangle, minimum width=0.5 cm, minimum height=0.5cm]
\tikzstyle{Circle1}=[fill=white, draw=black, shape=circle, minimum size=26 mm]

\tikzstyle{black dashed}=[-, draw=black, dashed]
\tikzstyle{new edge style 0}=[thick, ->]
\tikzstyle{thickedge}=[thick, -]

\usepackage{xcolor}
\usepackage{color}
\numberwithin{equation}{section}

\newtheorem{theorem}{Theorem}[section]
\newtheorem{lemma}[theorem]{Lemma}
\newtheorem{proposition}[theorem]{Proposition}
\newtheorem{corollary}[theorem]{Corollary}

\theoremstyle{definition}
\newtheorem{definition}[theorem]{Definition}

\newtheorem{question}[theorem]{Question}

\newtheorem{remark}[theorem]{Remark}

\newcommand{\ot}{\otimes}

\newcommand{\N}{\mathbb{N}}

\newcommand{\Z}{\mathbb{Z}}
\newcommand{\Ga}{\Gamma}
\DeclareMathOperator{\Ker}{Ker}

\DeclareMathOperator{\End}{End}
\DeclareMathOperator{\Hom}{Hom}

\newcommand{\parti}{\vdash}
\newcommand{\la}{\lambda}
\newcommand{\ol}{\overline}
\renewcommand{\S}{\mathbb{S}}
\DeclareMathOperator{\GL}{GL}
\DeclareMathOperator{\R}{R}
\DeclareMathOperator{\Ind}{Ind}
\newcommand{\Q}{\mathbb{Q}}

\DeclareMathOperator{\Res}{Res}
\newcommand{\Ainf}{K[X]}

\DeclareMathOperator{\Tr}{Tr}
\DeclareMathOperator{\Id}{Id}

\newcommand{\wt}{\widetilde}

\newcommand{\Ow}{\mathcal{O}}

\DeclareMathOperator{\Aut}{Aut}

\DeclareMathOperator{\Rep}{Rep}

\newcommand{\dash}{\text{\textemdash}}
\DeclareMathOperator{\Irr}{Irr}
\DeclareMathOperator{\Inf}{Inf}
\renewcommand{\H}{\mathcal{H}}
\newcommand{\W}{\mathcal{W}}
\newcommand{\F}{\mathbb{F}}
\DeclareMathOperator{\Stab}{Stab}

\title[Invariants and coverings]{Invariants that are covering spaces and their Hopf algebras}
\author{Ehud Meir}
\address{Institute of Mathematics, University of Aberdeen, Fraser Noble Building, Aberdeen AB24 3UE, UK}
  \email{ehud.meir@abdn.ac.uk}

\begin{document}
\maketitle
\begin{abstract}
In a previous paper by the author a universal ring of invariants for algebraic structures of a given type was constructed. This ring is a polynomial algebra that is generated by certain trace diagrams. 
It was shown that this ring admits the structure of a rational positive self adjoint Hopf algebra (abbreviated rational PSH-algebra), and was conjectured that it always admits a lattice that is a PSH-algebra, a structure that was introduced by Zelevinsky. In this paper we solve this conjecture, showing that the universal ring of invariants splits as the tensor product of rational PSH-algebras that are either polynomial algebras in a single variable, or admit a lattice that is a PSH-algebra. We do so by considering diagrams as topological spaces, and using tools from the theory of covering spaces. As an application we derive a formula that connects Kronecker coefficients with finite index subgroups of free groups and representations of their Weyl groups, and a formula for the number of conjugacy classes of finite index subgroup in a finitely generated group that admits a surjective homomorphism onto the group of integers. 
\end{abstract}
 
\section{Introduction}
In this paper we establish new relations between the invariant theory of $\GL_d(K)$, where $K$ is a field of characteristic zero, and the representation theory of finite groups. Recall that for $(p_i,q_i))\in (\N^2)^r$, an algebraic structure of type $((p_i,q_i))$ is a finite dimensional vector space $W$ equipped with structure tensors $x_i\in W^{p_i,q_i} = W^{\ot p_i}\ot(W^*)^{\ot q_i}$. In \cite{meirUIR} a universal ring of invariants $\Ainf$ for algebraic structures of type $((p_i,q_i))$ was constructed. The set $X$ is the union $\bigsqcup_{d\geq 0} X_d$, where $X_d$ consists of the isomorphism types of algebraic structures of type $((p_i,q_i))$ with a closed $\GL_d(K)$-orbit. See the Introduction and Section 3 of \cite{meirUIR} for more details. The ring $\Ainf$ is a polynomial ring that projects onto all the invariant rings $K[U_d]^{\GL_d}$, where $U_d$ is the affine space of structure constants for algebraic structures of type $((p_i,q_i))$ and dimension $d$, on which $\GL_d$ acts by a change of basis. In the case where the type of the algebraic structure is $((1,1),(1,1)\ldots, (1,1))$ our algebraic structure is a vector space equipped with a finite number of endomorphisms. The invariant rings for such structures were studied by Procesi, Drensky, and Razmyslov, among others. See \cite{ADS},\cite{BD},\cite{Hoge},\cite{Nakamoto},\cite{Popov},\cite{Procesi},\cite{Teranishi}. The invariant theory of other algebraic structures, such as Lie algebras, subfactors, Hopf algebras, and Hopf-Galois extensions, was also studied by Datt, Kodiyalam, and Sunder, by Kodiyalam and Sunder, by Bar Natan, by Vaintrob, and by the author of this paper. See \cite{DKS}, \cite{KS}, \cite{meirHopfI},\cite{meir2}, \cite{BN}, \cite{Va}

It was shown in \cite{meirUIR} that $\Ainf$ admits a structure of rational positive self adjoint Hopf algebra (rational PSH-algebra). This is a generalization of the structure of a PSH-algebra, defined by Zelevinsky in \cite{Zelevinsky}. PSH-algebras usually arise in the study of the representation theory of families of finite groups such as $(S_n)_n$ or $(\GL_n(\F_q))_n$, where $q$ is a prime power. They are graded, defined over $\Z$, and admit a basis with respect to which all the structure constants are non-negative integers. 

Zelevinsky showed that there is a unique, up to rescaling of the grading, \emph{universal} PSH-algebra, and that every PSH-algebra splits uniquely as the tensor product of universal PSH-algebras.
He also showed that the different tensor factors of universal PSH-algebras are in one-to-one correspondence with \textit{cuspidal} elements. By definition, these are basis elements of norm 1. Equivalently, these are basis elements that are orthogonal to all multiplications of basis elements of lower degree. The universal PSH-algebra is isomorphic to the PSH-algebra that arises in the representation theory of the symmetric groups. 

It was shown in \cite{meirUIR} that in the case of an algebraic structure that contains a single endomorphism, the universal ring of invariants admits a lattice that is a universal PSH-algebra. In other words, the universal ring of invariants $\Ainf$ contains a $\Z$-sub Hopf algebra $A$ that is free as a $\Z$-module, has a $\Z$-basis such that all the structure constants of multiplication and comultiplication with respect to this basis are non-negative integers, and such that the natural map $A\ot_{\Z} K\to \Ainf$ is an isomorphism of Hopf algebras. This lattice was used to give a new proof for the well known isomorphism $K[x_{ij}]^{\GL_d} = K[c_1,\ldots, c_d]$, where $\GL_d$ acts on $(x_{ij})_{i,j=1}^d$ by conjugation, and $c_i$ are the coefficients of the characteristic polynomial of $(x_{ij})_{i,j=1}^d$. The idea is that the PSH-algebra $A$ can be written as $A= \Z[Y_1,Y_2,\ldots]$, and if we write $I_d$ for the kernel of $\Ainf\to K[X_d]$ then it was shown that $I_d\cap A = (Y_{d+1},Y_{d+2},\ldots)$. So $A/(I_d\cap A)\cong \Z[Y_1,\ldots, Y_d]$. The elements $Y_i$ are then exactly the elements $c_i$ that generate $K[x_{ij}]^{\GL_d}$. Another important aspect of these invariants are that even though they are not $\Z$-linear combinations of diagram invariants, they still make sense over $\Z$ and therefore over any field. The aforementioned isomorphism $K[x_{ij}]^{\GL_d}\cong K[c_1,\ldots, c_d]$ remains valid in every characteristic. See Section \ref{sec:remarks} for further discussion. 
Since this PSH-algebra was useful to describe a concrete ring of invariants, the question if $\Ainf$ admits a lattice that is a PSH-algebra also in the general case was raised \cite[Question 1]{meirUIR}.  

In this paper we solve this question by constructing explicit PSH-algebras inside $\Ainf$. We will also give a representation-theoretical parametrisation of the cuspidal elements.

The first step that we will take will be to interpret the diagrams constructed in \cite{meirUIR} as graphs with a specific type of coloring. This will enable us to consider them as topological spaces. 
We will show that any color-preserving map between two such graphs is a covering, and that every graph covers a unique minimal (or \emph{irreducible}) graph. The first step in finding the lattice will be in proving the following result (see Proposition \ref{prop:splitting1}):

\begin{proposition}We have a tensor product decomposition $$\Ainf\cong \bigotimes_{\Ga_0} \Ainf_{\Ga_0},$$ where $\Ainf_{\Ga_0}$ is the polynomial algebra generated by all the connected graphs that cover $\Ga_0$. The tensor product is taken over all irreducible graphs. 
\end{proposition}
The family of covering spaces of $\Ga_0$ depends only on the fundamental group of $\Ga_0$. Since $\Ga_0$ is a graph, this is a free group of rank $m$ for some $m\geq 0$. The graded Hopf algebra $\Ainf_{\Ga_0}$ depends only on the parameter $m$. 

In case $m=0$ it holds that $\Ainf_{\Ga_0}= K[\Ga_0]$, a polynomial algebra on the single variable $\Ga_0$. 
In case $m>0$ we will show that $\Ainf_{\Ga_0}$ admits a lattice $\H_{\Ga_0}$ which gives it a structure of a PSH-algebra, and we will describe explicitly its cuspidal elements. To do so, we need some more definitions.

Write $F_m:= \langle z_1,\ldots, z_m\rangle$ for the free group of rank $m$, and denote by $\upsilon:F_m\to \Z$ the group homomorphism that sends $z_i$ to 1 for every $i$. Write $F_m^0=\Ker(\upsilon)$. 
The element $z_1$ acts by conjugation on the set of all finite transitive $F_m^0$-sets. We will call a finite transitive $F_m^0$-set $\Ow$ \textit{strongly finite}, if $z_1^k(\Ow)\cong \Ow$ for some $k>0$. We write $k(\Ow)$ for the minimal such $k$. 
For every strongly finite transitive $F_m^0$-set $\Ow$, we fix an isomorphism $\Phi(\Ow):z_1^{k(\Ow)}(\Ow)\to \Ow$. Conjugation by $\Phi(\Ow)$ then induces an automorphism $\rho(\Ow)$ on $\Aut_{F_m^0}(\Ow)$ (see Equation \ref{eq:defrho}). We write $FT(F_m,F_m^0)$ for the set of all pairs of the form $(\Ow,\ol{[W]})$, where $\Ow$ is a strongly finite transitive $F_m^0$-set, $W$ is an irreducible representation of $\Aut_{F_m^0}(\Ow)$, and $\ol{[W]}=\ol{[W']}$ if and only if $W$ and $W'$ are conjugate under the action of $\rho$. We write $l(W)$ for the cardinality of the orbit of $W$ under the action of $\rho$.  
Theorems \ref{thm:main} and \ref{thm:cuspidals} combine to give the following result: 
\begin{theorem}
The lattice $\H_{\Ga_0}$ constructed in Section \ref{sec:mains} is a PSH-algebra, and the set of cuspidal elements of this algebra is in one-to-one correspondence with $FT(F_m,F_m^0)$. The cuspidal element that corresponds to the orbit of $(\Ow,\ol{[W]})$ has degree $|\Ow|k(\Ow)l(\Ow)$. 
\end{theorem}
We will give some examples of cuspidal elements in the cases $m=1$ and $m>1$ in Section \ref{sec:examples}.

The construction presented in this paper has further consequences to general finitely generated groups. 
Let $G$ be a finitely generated group that splits as a semidirect product $G=G^0\rtimes \Z$. The definition of $FT(F_m,F_m^0)$ generalizes directly to the definition of $FT(G,G^0)$. In Section \ref{sec:fggrp} we will prove the following result (see Proposition \ref{prop:fggrp}):
\begin{proposition}
Under the above assumptions on $G$, the number of conjugacy classes of subgroups of $G$ of index $n$ is equal to $\sum_{d|n} |FT(G,G^0)_d|$, where $FT(G,G^0)_d$ is the set of elements in $FT(G,G^0)$ of degree $d$.
\end{proposition}
We will show some concrete calculations with the above formula for finitely generated abelian groups and for Baumslag-Solitar groups in Section \ref{sec:fggrp}.

In \cite[Section 9]{meirUIR} a very complicated formula for the Hilbert series of $\Ainf$ was derived, using both the Littlewood-Richardson coefficients and Kronecker coefficients. In this paper we derive a simpler formula for the Hilbert series of $\Ainf_{\Ga_0}$. We will prove in Section \ref{sec:hilbert} the following result:
\begin{theorem} Assume that $\pi_1(\Ga_0,v) = F_m$. The Hilbert series of $\Ainf_{\Ga_0}$ can be expressed in the following equivalent ways:
$$\sum_{n\geq 0}\dim((\Ainf_{\Ga_0})_n)X^n = \sum_n\sum_{\mu,\la_1,\ldots, \la_m,\mu\parti n}g(\la_1,\ldots,\la_m,\mu)^2X^n =$$ $$ \prod_{\Ow\in T(F_m)}\frac{1}{1-X^{|\Ow|}} = \prod_{(\Ow,\ol{[W]})\in FT(F_m,F_m^0),n}\frac{1}{1-X^{|\Ow|k(\Ow)l([W])n}},$$
where $T(F_m)$ is a set of representatives of the isomorphism classes of finite transitive $F_m$-sets, and $g(\la_1,\ldots,\la_m,\mu)$ are the iterated Kronecker coefficients, see  Subsection \ref{subsec:linalgrep}. 
\end{theorem}

\section{Preliminaries}\label{sec:prelim}
Throughout this paper we assume that $K$ is an algebraically closed field of characteristic zero that is equipped with an automorphism $z\mapsto \ol{z}$ of order 2 that inverts all roots of unity. (e.g. $K=\ol{\Q}$ or $K=\mathbb{C}$). We will need these assumptions when applying results from the representation theory of finite groups. 
\subsection{Universal rings of invariants} 
Fix a tuple $((p_i,q_i))$ in $(\N^2)^r$. Recall that an algebraic structure of type $((p_i,q_i))$ is a vector space $W$ equipped with structure tensors $x_i\in W^{p_i,q_i} = W^{\ot p_i}\ot (W^*)^{q_i}$. In \cite{meirUIR} we constructed a universal ring of invariants $K[X]_{((p_i,q_i))}=\Ainf$ for algebraic structures of type $((p_i,q_i))$. This ring is a polynomial algebra on infinitely many variables. These variables are 
all the closed connected diagrams formed by boxes of types $x_1,\ldots, x_r$, where $x_i$ has $q_i$ input strings and $p_i$ output strings, and all the input strings are connected to all the output strings. Such a diagram represents an invariant of the form $\Tr(L_{\sigma}x_1^{_\ot a_1}\ot\cdots\ot x_r^{\ot a_r})$, where $n = \sum a_i p_i = \sum a_iq_i$, $\sigma\in S_n$ is a permutation, and $L_{\sigma}: W^{\ot n}\to W^{\ot n}$ is given by permuting the tensors in $S^{\ot n}$ according to $\sigma$. See Equation 1 in \cite{meirUIR} for more details. 

We will fix the type $((p_i,q_i))$ throughout the paper and write $\Ainf$ for the universal ring of invariants. 
Let $d\in \N$. The structure constants with respect to the standard basis of $K^d$ of a structure of type $((p_i,q_i))$ form an affine space $U_d$. The algebraic group $\GL_d$ acts on $U_d$ by a change of basis, and the different $\GL_d$-orbits in $U_d$ correspond to the different isomorphism types of structures of dimension $d$. The diagrams in $\Ainf$ can be interpreted as $\GL_d$-invariant polynomials on $U_d$. This gives us a map $\Ainf\to K[U_d]^{\GL_d}$ that turns out to be surjective \cite[Section 3]{meirUIR}. We write $I_d$ for the kernel of this map. 

We have seen in \cite{meirUIR} that the algebra $\Ainf$ has a richer structure of an $\N^r$-graded Hopf algebra, where the connected diagrams are primitive, and that moreover it is equipped with an inner product $\langle-,-\rangle$ that makes the multiplication dual to the comultiplication. The pairing is given by $\langle Di_1,Di_2\rangle = 0$ if $Di_1\neq Di_2$, and $\langle Di,Di\rangle = |\Aut(Di)|$, the cardinality of the automorphism group of the diagram $Di$. In this paper we will alter the inner product $\langle-,-\rangle$ to be a sesquilinear form, in order to be able to apply tools from the representation theory of finite groups. We thus have
$$\langle\sum_i a_i Di_i,\sum_i b_i Di_i\rangle = \sum_i a_i\ol{b_i}|\Aut(Di_i)|, a_i,b_i\in K.$$ 
%We will find an integral basis for this algebra, on which the form $\langle-,-\rangle$ will be symmetric. 
We recall that the automorphism group of a diagram is the group of all permutations on the boxes in the diagram that leaves the diagram stable. For example, if $T$ is a box with one input and one output string, then the automorphism group of the diagram 
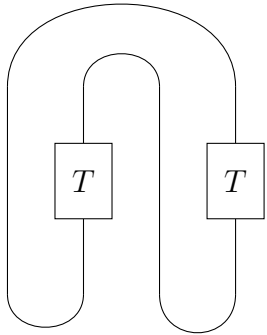
\begin{center}
\begin{tikzpicture}
	\begin{pgfonlayer}{nodelayer}
		\node [style=1function] (0) at (2.75, 2) {$T$};
		\node [style=1function] (1) at (0.75, 2) {$T$};
		\node [style=none] (2) at (2.75, 2.25) {};
		\node [style=none] (3) at (2.75, 3.25) {};
		\node [style=none] (4) at (0.75, 2.25) {};
		\node [style=none] (5) at (0.75, 3.25) {};
		\node [style=none] (6) at (0.75, 1.75) {};
		\node [style=none] (7) at (0.75, 0.5) {};
		\node [style=none] (8) at (2.75, 1.75) {};
		\node [style=none] (9) at (2.75, 0.5) {};
		\node [style=none] (10) at (1.75, 3.25) {};
		\node [style=none] (11) at (1.75, 0.5) {};
		\node [style=none] (12) at (-0.25, 3.25) {};
		\node [style=none] (13) at (-0.25, 0.5) {};
	\end{pgfonlayer}
	\begin{pgfonlayer}{edgelayer}
		\draw (6.center) to (7.center);
		\draw (4.center) to (5.center);
		\draw [bend left=90, looseness=1.50] (5.center) to (10.center);
		\draw (10.center) to (11.center);
		\draw [bend right=90, looseness=1.75] (11.center) to (9.center);
		\draw (9.center) to (8.center);
		\draw (2.center) to (3.center);
		\draw [bend left=270, looseness=1.25] (3.center) to (12.center);
		\draw (12.center) to (13.center);
		\draw [bend right=90, looseness=1.50] (13.center) to (7.center);
	\end{pgfonlayer}
\end{tikzpicture}
\end{center}
representing the invariant $\Tr(T^2)$, is cyclic of order 2 (and, more generally, the automorphism group of the diagram representing $\Tr(T^n)$ is cyclic of order $n$). 

\subsection{PSH algebras}
We recall the following definition of PSH-algebras from \cite{Zelevinsky} and of rational PSH-algebras from \cite{meirUIR}: 
\begin{definition}
A positive self-adjoint Hopf algebra (or PSH-algebra) is an $\N$-graded $\Z$-Hopf algebra $A$ equipped with a graded basis $B$ of $A$ and a pairing $\langle-,-\rangle:A\ot_{\Z} A\to \Z$ such that the following conditions are satisfied:
\begin{enumerate}
\item The basis $B$ is orthonormal with respect to $\langle-,-\rangle$. In other words- for every $x,y\in B$ we have $\langle x,y\rangle = \delta_{x,y}$.
\item The multiplication is adjoint to the comultiplication with respect to $\langle -,- \rangle$ where $A\ot_{\Z}A$ has the tensor product pairing. 
\item The unit $u:\Z\to A$ and the counit $\epsilon:A\to\Z$ are adjoint with respect to $\langle -,-\rangle$ where $\Z$ has the canonical pairing.
\item The algebra $A$ is connected, that is $A_0=\Z$.
\item All the structure constants of $m,\Delta,u,\epsilon$ with respect to the basis $B$ are non-negative integers.
\end{enumerate}
By a graded basis we mean that $B=\sqcup_n B_n$, where $B_n$ is a basis for $A_n$. 
\end{definition}
\begin{definition} A rational $K$-PSH-algebra is an $\N^r$-graded $K$-Hopf algebra $A$ equipped with a graded basis $B$ that satisfies the following conditions:
\begin{enumerate}
\item The basis $B$ is orthogonal and positive with respect to $\langle-,-\rangle$. In other words- for every $x,y\in B$ we have $\langle x,y\rangle = \delta_{x,y}c(x)$ for some $c(x)\in\Q_{+}.$
\item The multiplication is adjoint to the comultiplication with respect to $\langle -,- \rangle$, where $A\ot_K A$ is equipped with the tensor product pairing. 
\item The unit $u:K\to A$ and the counit $\epsilon:A\to K$ are adjoint with respect to $\langle -,-\rangle$, where $K$ has the canonical pairing.
\item The algebra $A$ is connected, that is $A_0=K$.
\item All the structure constants of $m,\Delta,u,\epsilon$ with respect to the basis $B$ are in $\Q_{+}$.
\end{enumerate}
The number $r$ which appears in the grading is some positive integer. 
\end{definition}
So every PSH-algebra gives a rational PSH-algebra by extension of scalars. Finding a lattice that is a PSH-algebra inside a $K$-rational PSH algebra is more difficult. 

Zelevinsky proved a very strong classification theorem for PSH-algebras. He showed that, up to isomorphism and rescaling of the grading, there is one basic (or \textit{universal}) PSH-algebra \cite[p. 27]{Zelevinsky}, and that every PSH-algebra $\H$ splits in a unique way as the tensor product $\H=\bigotimes_{\rho} \H_{\rho}$ of universal PSH-algebras \cite[p 22]{Zelevinsky}. Here the tensor product is taken over all the \textit{cuspidal elements} of $\H$. These are basis elements that satisfy $||\rho||=1$ and $\Delta(\rho) = \rho\ot 1+ 1\ot \rho$. The algebra $\H_{\rho}$ is then equal to $$span_{\Z}\{b\in B | \langle b,\rho^n\rangle\neq 0\text{ for some }n\}.$$
The universal PSH-algebra is $\Z[X_1,X_2,\cdots]$ where $\deg(X_n)=n$, $\Delta(X_n) = \sum_{a+b=n}X_a\ot X_b$ and $||X_n||^2=1$. This algebra can be interpreted as $\bigoplus_n \R(S_n)$, where $R(S_n)=K_0(S_n)$ is the character group of $S_n$. The irreducible representations of $S_n$ form a $\Z$-basis of $R(S_n)$, and the multiplication and comultiplication arise from induction and restriction along the inclusions $S_n\times S_m\to S_{n+m}$.  We will denote the universal PSH-algebra by $Zel$. 
If $\rho\in \H$ is a cuspidal element, then $\H_{\rho}= \Z[X_1^{\rho},X_2^{\rho},\ldots]$ where $\deg(X_n^{\rho}) = n\deg(\rho)$. The following lemma then follows immediately from the structure of PSH-algebras:
\begin{lemma}\label{lem:numofgens}
Let $\H$ be a PSH-algebra. Let $a_n$ be the number of cuspidal elements of degree $n$ in $\H$. Then $\H$ is a graded polynomial algebra. The number of variables of degree $n$ in $\H$ is equal to $\sum_{d|n} a_d$. 
\end{lemma}

\subsection{Clifford Theory}
We will use some parts of Clifford Theory in this paper. For more details, see Chapter 12 of \cite{Kar}.
For a finite group $G$, we write $\Irr(G)$ for the set of isomorphism types of irreducible $G$-representations. We will use the natural identification of this set with the set of irreducible characters of $G$. 
Assume that $G$ fits into a short exact sequence 
$$1\to N\to G\to Q\to 1.$$
For every $g\in G$, conjugation by $g$ is an automorphism of $N$. We write $g^*:\Irr(N)\to \Irr(N)$ for the bijection given by pre-composing with conjugation by $g$. Since $N$ acts trivially on $\Irr(N)$ the action of $G$ on this set factors through $G\to Q$. 
For $\phi\in\Irr(N)$ we write $\Irr(G)_{\phi}$ for the irreducible representations of $G$ whose restriction to $N$ contains only representations in the orbit $G\cdot \phi$.
Let $\phi_1,\ldots, \phi_t$ be a set of representatives of the $G$-orbits in $\Irr(N)$. The first assertion of Clifford Theory is the following:
\begin{equation}\tag{C1}\label{C1}\Irr(G) = \bigsqcup_{i=1}^t \Irr(G)_{\phi_i}.\end{equation}

Consider now the subgroup $\Stab_G(\phi)$. It contains $N$ as a normal subgroup, and the $\Stab_G(\phi)$-orbit of $\phi$ is just $\{\phi\}$. For $\psi\in \Irr(\Stab_G(\phi))_{\phi}$ consider the $G$-representation $\Ind^G_{\Stab_G(\phi)}\psi$. The second assertion of Clifford Theory is the following:
\begin{equation}\tag{C2}\label{C2}\text{The induction functor induces a bijection } \Irr(\Stab_G(\phi))_{\phi}\to \Irr(G)_{\phi}.\end{equation}

Let $W$ be the irreducible representation of $N$ with character $\phi$. Schur's lemma implies that we can extend the action of $N$ to a projective action of $\Stab_G(\phi)$. We can thus think of $W$ as a representation of the twisted group algebra $K^{\beta}\Stab_G(\phi)$, where $[\beta]\in H^2(\Stab_G(\phi),K^{\times})$. This algebra has a basis $\{u_g\}_{g\in \Stab_G(\phi)}$ and multiplication $$u_{g_1}u_{g_2} = \beta(g_1,g_2)u_{g_1g_2}.$$
The class of $\beta$ is inflated from a class in $H^2(\Stab_Q(\phi),K^{\times})$ that we shall denote by the same letter. If $V$ is a representation of $K^{\beta^{-1}}\Stab_Q(\phi)$ then by inflation it is also a representation of $K^{\beta^{-1}}\Stab_G(\phi)$. The cocycles $\beta$ and $\beta^{-1}$ cancel each other in the diagonal action of $\Stab_G(\phi)$ on $W\ot V$, and we get a proper representation of $\Stab_G(\phi)$. We denote by $\Irr(K^{\beta^{-1}}\Stab_Q(\phi))$ the set of equivalence classes of irreducible representations of the algebra $K^{\beta^{-1}}\Stab_Q(\phi)$. The third assertion of Clifford Theory is the following, where we denote by $[W]$ the isomorphism class of a representation $W$:
\begin{equation}\tag{C3}\label{C3}\text{The map } \Irr(K^{\beta^{-1}}\Stab_Q(\phi))\to \Irr(\Stab_G(\phi))_{\phi} \text{ given by } [V]\mapsto [W\ot V] \text{ is a bijection}\end{equation}
In this paper we will apply this part of Clifford Theory only in case the cocycle $\beta$ is trivial. We record this in the following:
\begin{equation}\tag{C3'}\label{C3'}\text{ Assume that } W \text{ admits a structure of a } \Stab_G(\phi) \text{ representation. Then } [\beta]=1 \end{equation} $$\text{ and we therefore have bijections } \Irr(\Stab_Q(\phi))\to \Irr(\Stab_G(\phi))_{\phi}\to \Irr(G)_{\phi}$$
\begin{remark} We can think of $\Irr(N)$ as the set of isomorphism classes of irreducible representations of $N$, or as the set of irreducible characters of $N$. In the sequel we will not make a strict distinction between these two descriptions of $\Irr(N)$.
\end{remark}

\subsection{Linear algebra and representation theory}\label{subsec:linalgrep}
The next few lemmas will be of use when calculating a basis for $\Ainf$.
\begin{lemma}\label{lem:cyclic}
Let $V$ be a finite dimensional vector space, and let $L:V^{\ot n}\to V^{\ot n}$ be given by $L(u_1\ot\cdots\ot u_n) = u_n\ot u_1\ot\cdots\ot u_{n-1}$. Then $\Tr(L)= \Tr(L^{-1})= \dim(V)$.
\end{lemma}
\begin{proof}
Fix a basis $\{v_i\}$ of $V$. Then $\{v_{i_1}\ot\cdots\ot v_{i_n}\}$ is a basis for $V^{\ot n}$ that $L$ permutes. 
The trace of $L$ is thus the number of basis elements that $L$ fixes. These are exactly the elements $v_i\ot v_i\ot\cdots\ot v_i$, and there are exactly $\dim(V)$ of them. The proof for $\Tr(L^{-1})=\dim(V)$ is exactly the same. 
\end{proof}

Let $G_1,\ldots, G_n$ be finite groups. We have a bijection \begin{equation}\label{eq:product}\Irr(G_1)\times\cdots\times \Irr(G_n)\cong \Irr(G_1\times\cdots\times G_n)\end{equation} $$([V_1],\ldots, [V_n])\mapsto [V_1\boxtimes V_2\boxtimes\cdots\boxtimes V_n]$$
where $V_1\boxtimes\cdots\boxtimes V_n= V_1\ot \cdots\ot V_n$ is a $G_1\times\cdots\times G_n$-representation by $$(g_1,\ldots, g_n)\cdot (v_1\ot\cdots\ot v_n) = g_1\cdot v_1\ot\cdots\ot g_n\cdot v_n.$$
For a single group $G$, we will use the fact that if $V_1,\ldots, V_n$ are $G$-representations then the restriction of $V_1\boxtimes\cdots\boxtimes V_n$ to the diagonal subgroup $G\subseteq G^n$, is just the representation $V_1\ot\cdots\ot V_n$ with the diagonal $G$-action.  

We use Chapter 2 of \cite{Sagan} as a reference for the representation theory of the symmetric groups. There is a known bijection between partitions of $n$ and isomorphism classes of irreducible representations of $S_n$. We will denote this bijection by $\la\mapsto \S_{\la}$, and call $\S_{\la}$ the Specht module that corresponds to $\la$. We have $$\Ind_{S_a\times S_b}^{S_{a+b}}\S_{\la}\boxtimes \S_{\mu} \cong \bigoplus \S_{\nu}^{\oplus c_{\la,\mu}^{\nu}},$$ where $c_{\la,\mu}^{\nu}$ are the Littlewood-Richardson coefficients. 
\begin{definition} Assume that $\sum_{i =1}^b a_i=n$. We write $S_{a_1,\ldots, a_b}$ for the subgroup of $S_n$ that stabilizes the sets $\{1,\ldots, a_1\},\{a_1+1,\ldots, a_1+a_2\},\ldots, \{(\sum_{i=1}^{b-1}a_i) + 1,\ldots, n\}$.
\end{definition}
The subgroup $S_{a_1,\ldots, a_b}$ is naturally isomorphic to $S_{a_1}\times\cdots\times S_{a_b}$. 

The coefficients for the multiplication in the ring $\R(S_n)$ are called the \textit{Kronecker coefficients}. We write them as $g(\la_1,\la_2,\mu)$. Thus, for two partitions $\la_1,\la_2$ of $n$ we have 
$$\S_{\la_1}\ot \S_{\la_2} = \bigoplus_{\mu\parti n}\S_{\mu}^{\oplus g(\la_1,\la_2,\mu)}.$$ 
More generally, for partitions $\la_1,\ldots, \la_m,\mu\parti n$ we will write $$g(\la_1,\ldots,\la_m,\mu)=\dim\Hom_{S_n}(\S_{\la_1}\ot\cdots\ot \S_{\la_m},\S_{\mu}).$$ Thus, 
$$\S_{\la_1}\ot\cdots\ot \S_{\la_m} = \bigoplus_{\nu\parti n}\S_{\nu}^{\oplus g(\la_1,\ldots, \la_m,\nu)}.$$ 
We call $g(\la_1,\ldots, \la_m,\mu)$ the iterated Kronecker coefficients \cite[Section 2]{meirUIR}. 

\subsection{Wreath products}
Let $B$ be a finite group, and let $n$ be an integer. We consider the wreath product, $S_n\ltimes B^n$, where $S_n$ acts on $B^n$ by permuting the factors. 
This group fits into the short exact sequence $$1\to B^n\to S_n\ltimes B^n\to S_n\to 1.$$ We can thus use Clifford Theory to describe the representation theory of this group. To do so, write $\Irr(G) = \{[W_1],\ldots, [W_b]\}$.
\begin{lemma}
Isomorphism classes of irreducible representations of $S_n\ltimes B^n$ are in one to one correspondence with tuples $(a_1,\ldots, a_b,\la_1,\ldots, \la_b)$ such that $a_i\in \N$, $\la_i\parti a_i$ and $\sum_i a_i = n$.
\end{lemma}
\begin{proof}
We use Clifford Theory. We first describe the irreducible representations of $B^n$. By Equation \ref{eq:product} these representations are all of the form $W_{i_1}\boxtimes\cdots\boxtimes W_{i_n}$ for some $i_1,i_2,\ldots i_n\in \{1,\ldots, b\}$. Next, we considser the action of $S_n$ on $\Irr(B^n)$. Let $\tau\in S_n$. The $B^n$ representation $\tau^*(W_{i_1}\ot\cdots\ot W_{i_n})$ is equal to $W_{i_1}\boxtimes\cdots\boxtimes W_{i_n}$ as a vector space, and $(g_1,\ldots, g_n)\in B^n$ acts by $\tau(g_1,\ldots, g_n)\tau^{-1} = (g_{\tau^{-1}(1)},\ldots, g_{\tau^{-1}(n)})$. By considering the action of the different factors of $B$ in $B^n$ we see that this representation is isomorphic to $W_{i_{\tau(1)}}\boxtimes\cdots\boxtimes W_{i_{\tau(n)}}$. 
So the action of $S_n$ on $\Irr(B_n)$ is given by shuffling the representations. 
This already shows us that any $S_n$-orbit in $\Irr(B^n)$ contains a unique representation of the form $W_1^{\boxtimes a_1}\boxtimes W_2^{\boxtimes a_2}\boxtimes\cdots\boxtimes W_b^{\boxtimes a_b}$ with $\sum_i a_i=n$, where $a_i=|\{j| i_j=i\}|$.

The stabilizer in $S_n$ of the above representation is then $S_{a_1,\ldots, a_b}$. The representation $W_1^{\boxtimes a_1}\boxtimes\cdots\boxtimes W_b^{\boxtimes a_b}$ is also a representation of $S_{a_1,\ldots, a_b}\ltimes B^n$, where $S_{a_1\ldots a_b}$ acts by tensor-permuting the tensorands. This shows that the two-cocycle arising in Clifford Theory is trivial here, and there is a one-to-one correspondence between irreducible representations of $S_n\ltimes B^n$ lying over the orbit of $W_1^{\boxtimes a_1}\boxtimes\cdots\boxtimes W_b^{\boxtimes a_b}$ and irreducible representations of $S_{a_1,\ldots, a_b}$. Since the irreducible representations of $S_{a_i}$ are in one to one correspondence with partitions of $a_i$ we get the result.
\end{proof}
Using Clifford Theory again, we see that the irreducible representation of $S_n\ltimes B^n$ that corresponds to $(a_1,\ldots, a_b,\la_1,\ldots,\la_b)$ is given by 
$$\Ind_{S_{a_1,\ldots, a_b}\ltimes B^n}^{S_n\ltimes B^n}((\Inf_{S_{a_1,\ldots a_b}}^{S_{a_1,\ldots, a_b}\ltimes B^n} \S_{\la_1}\boxtimes\cdots\boxtimes \S_{\la_b})\ot (W_1^{\boxtimes a_1}\boxtimes\cdots\boxtimes W_b^{\boxtimes a_b})).$$
\begin{definition}\label{def:Wala}
We denote the above representation by $W_{(a_i,\la_i)}$. 
\end{definition}
Assume now that $\rho:B\to B$ is an automorphism. Then $\rho^n:B^n\to B^n, (b_1,\ldots, b_n)\mapsto (\rho(b_1),\ldots,\rho(b_n))$ is an automorphism of $B^n$, and $\xi:S_n\ltimes B^n\to S_n\ltimes B^n$ $(\sigma,b_1,\ldots, b_n)\mapsto (\sigma,\rho(b_1),\ldots, \rho(b_n))$ is an automorphism of the wreath product $S_n\ltimes B^n$. The action of $\rho$ on $B$ induces an action on $\Irr(B)$. We denote this action by $\rho$ as well, so that $\rho^*(W_i)\cong W_{\rho(i)}$. We claim the following:
\begin{lemma} The action of $\xi$ on $\Irr(S_n\ltimes B^n)$ is given by $\xi^*(W_{(a_i,\la_i)}) = W_{(a_{\rho^{-1}(i)},\la_{\rho^{-1}(i)})}$. As a result, $W_{(a_i,\la_i)}$ is $\xi$-invariant if and only if $a_i=a_j$ and $\la_i=\la_j$ whenever $i$ and $j$ are in the same $\rho$-orbit. 
\end{lemma}
\begin{proof}
We have $$(\rho^n)^*(W_1^{\boxtimes a_1}\boxtimes\cdots\boxtimes W_b^{\boxtimes a_b})\cong W_{\rho(1)}^{\boxtimes a_1}\boxtimes\cdots\boxtimes W_{\rho(b)}^{\boxtimes a_b}$$ as $B^n$-representations. After re-ordering the factors in the last tensor product we get the representation $$\W:= W_1^{\boxtimes a_{\rho^{-1}(1)}}\boxtimes\cdots\boxtimes W_b^{\boxtimes a_{\rho^{-1}(b)}}.$$
We thus see that taking the $\W$-isotypic component of $\xi^*(W_{(a_i,\la_i)})$ gives $\S_{\rho^{-1}(1)}\boxtimes\cdots\boxtimes \S_{\rho^{-1}(b)}\boxtimes \W$. So we get that $$\xi^*(W_{(a_i,\la_i)}) = W_{(a_{\rho^{-1}(i)},\la_{\rho^{-1}(i)})}$$ as desired. The claim about $\xi$-invariant representations is now immediate. 
\end{proof}

 For the next lemma, we consider a semi-direct product $G=Q\ltimes N$. Let $\W$ be an irreducible representation of $N$, and let $Q_1<Q$ be a subgroup. Write $Q_2 = \Stab_Q([\W])$ and $Q_3 = Q_1\cap Q_2$. Assume that $W$ can be extended to a $Q_2\ltimes N$-representation, and fix such an extension. We claim the following:
\begin{lemma}\label{lem:technical}
Let $U$ be an irreducible representation of $Q_3$ and let $V$ be an irreducible representation of $Q_2$. 
Then $\wt{U}=\Ind_{Q_3\ltimes N}^{Q_1\ltimes N}(U\ot \W)$ is an irreducible representation of $Q_1\ltimes N$, $\wt{V}=\Ind_{Q_2\ltimes N}^{Q\ltimes N}(V\ot \W)$ is an irreducible representation of $G$, and 
$$\Hom_{Q_3}(U,\Res^{Q_2}_{Q_3}V)\to \Hom_{Q_1\ltimes N}(\wt{U},\Res^G_{Q_1\ltimes N}\wt{V})$$ 
$$ T\mapsto \wt{T}$$
where $\wt{T}(g\ot u\ot w) = g\ot T(u)\ot w$ is an isomorphism. As a result, the $\wt{U}$-isotypic component of $\Res^G_{Q_1\ltimes N}\wt{V}$ is equal to $(KQ_1\ltimes N)\ot_{Q_3\ltimes N} V_U\ot \W\subseteq \wt{V}$, where $V_U$ is the $U$-isotypic component of $\Res^{Q_2}_{Q_3}V$. 
\end{lemma}
\begin{proof}
The fact that $\wt{U}$ and $\wt{V}$ are irreducible follows from Clifford Theory together with the fact that $Q_2$ is the stabilizer of $[\W]$ in $Q$, and the fact that $Q_3 = Q_1\cap Q_2$ is the stabilizer of $[\W]$ in $Q_1$. 
By using the definitions of $\wt{U}$ and $\wt{V}$ Frobenius reciprocity gives us 
$$\Hom_{Q_1\ltimes N}(\wt{U},\Res^G_{Q_1\ltimes N}\wt{V})\cong \Hom_G(\Ind_{Q_1\ltimes N}^G \Ind_{Q_3\ltimes N}^{Q_1\ltimes N}(U\ot \W),\Ind_{Q_2\ltimes N}^G (V\ot \W))\cong$$
$$\Hom_G(\Ind_{Q_3\ltimes N}^G(U\ot \W),\Ind_{Q_2\ltimes N}^G (V\ot \W))= \Hom_G(\Ind_{Q_2\ltimes N}^G \Ind_{Q_3\ltimes N}^{Q_2\ltimes N}(U\ot \W),\Ind_{Q_2\ltimes N}^G (V\ot \W)).$$
We have an isomorphism $$\Ind_{Q_3\ltimes N}^{Q_2\ltimes N} (U\ot \W)\cong \Ind_{Q_3}^{Q_2}(U)\ot \W$$ $$g\ot u\ot w\mapsto g\ot u\ot w,$$ and the last hom-space is thus isomorphic to 
$$\Hom_G(\Ind_{Q_2\ltimes N}^G (\Ind_{Q_2\ltimes N}^G\Ind_{Q_3}^{Q_2}(U)\ot \W),\Ind_{Q_2\ltimes N}^G (V\ot \W))\stackrel{\text{\ding{172}}}{\cong} $$
$$\Hom_{Q_2\ltimes N}(\Ind_{Q_3}^{Q_2}(U)\ot\W,V\ot \W)\stackrel{\text{\ding{173}}}{\cong} \Hom_{Q_2}(\Ind_{Q_3}^{Q_2}(U),V)\stackrel{\text{\ding{174}}}{\cong} \Hom_{Q_3}(U,\Res^{Q_2}_{Q_3}V).$$
We used (C2) for the isomorphism \ding{172} and (C3) for the isomorphism \ding{173}. The isomorphism \ding{174} follows from Frobenius reciprocity. Following all the isomorphisms we had here we get the isomorphism stated in the lemma. The last statement, about the isotypic components, follows by considering the sum of all the images of maps of the form $\wt{T}$.
\end{proof}
 
\section{Translating diagrams into graphs}
We keep the type $((p_i,q_i))_{i=1}^r$ fixed. 
%Recall that diagrams are made of boxes labelled by $x_i$, where $i\in \{1,\ldots,r\}$, and connecting strings. A box with the label $x_i$ has $q_i$ input strings and $p_i$ output strings. A diagram is closed if every output string is connected to some input string and vice versa.
%The algebra $\Ainf$ has a basis given by all isomorphism classes of closed diagrams, and multiplication is given by taking disjoint unions. The algebra $\Ainf$ is a polynomial algebra on the set of closed connected diagrams. 
The main object of study in this paper will be the closed diagrams for structures of type $((p_i,q_i))$. Such a diagram is made of boxes labeled by elements in $\{x_i\}_{i=1}^r$, and strings, where a box with a label $x_i$ has $p_i$ ordered output strings and $q_i$ ordered input strings. All input strings are connected to output strings and all output strings are connected to input strings. We will now explain how to translate such diagrams into graphs with a coloring. 

To this end, a graph $\Ga=(V,E)$ is given by a set of vertices $V$ and a set of edges $E$, equipped with two maps $s,t:E\to V$, indicating that the edge $e$ goes from $s(e)$ to $t(e)$. Multiple edges and self edges are allowed. 
In this paper we will consider graphs with colorings $c:V\to \N$, $c_o,c_i:E\to \N$. We call $c_o$ the output coloring and $c_i$ the input coloring. We will also write $c(e) = (c_o(e),c_i(e))$.
\begin{definition}
A graph $\Ga=(V,E)$ with a coloring $c$ is called adequate if: 
\begin{enumerate}
\item For every $v\in V$ we have $c(v)\in \{1,\ldots, r\}$. 
\item If $c(v)=a$ then $|\{e\in E| s(e)=v\}|=p_a$, $|\{e\in E| t(e)=v\}|=q_a$, $\{c_o(e)\}_{s(e)=v} = \{1,2,\ldots, p_a\}$, and $\{c_i(e)\}_{t(e)=v} = \{1,2,\ldots, q_a\}$.
\end{enumerate}
\end{definition}
There is an obvious notion of morphisms between graphs:
\begin{definition} 
A morphism of graphs $\phi:\Ga=(V,E)\to \Ga'=(V',E')$ is a pair of maps $\phi_V:V\to V'$ and $\phi_E:E\to E'$ such that $\phi_V(s(e)) = s(\phi_E(e))$, $\phi_V(t(e))= t(\phi_E(e))$ for every $e\in E$. If $\Ga$ and $\Ga'$ have colorings then the morphism $\phi$ is called chromatic if it preserves the colorings, that is: $(c_o(\phi(e)),c_i(\phi(e))) = (c_0(e),c_i(e))$ and $c(\phi(v))= \phi(v)$ for every $e\in E$ and $v\in V$.
\end{definition}
We now claim the following:
\begin{lemma}
There is a one-to-one correspondence between isomorphism classes of finite adequate graphs and closed diagrams. 
\end{lemma}
\begin{proof}
If $Di$ is a closed diagram, replace every box $x_a$ by a vertex with color $a$. If the $i$-th output string of an $x_a$ box is connected to the $j$-th input string of an $x_b$ box, replace this by an edge $e$ with $c_i(e)= j$, $c_o(e)=i$. It is clear that we get a bijection this way, and that this respects isomorphisms.
\end{proof}
\begin{remark} In principle, it is possible to develop all of the results in this paper by just using diagrams, without mentioning graphs. However, since the language of graphs is much better developed, and it is easier to think of them as topological spaces, we prefer to use it. Also, it makes more sense to speak about infinite graphs, which we will have to do later when speaking about universal coverings. The correspondence presented here also gives an isomorphism between the automorphism group of a diagram and of the colored graph that represents it. 
We are thus going to work with adequate graphs instead of diagrams from now on. 
\end{remark}
\begin{definition}
A morphism of graphs $\Ga\to \Ga'$ is called a covering if it is surjective and a local homeomorphism, where we identify the graphs with their topological realizations. 
\end{definition}
We claim the following:
\begin{lemma} Every chromatic morphism $\phi:\Ga=(V,E)\to \Ga'=(V',E')$ between adequate graphs in which $\Ga'$ is connected is a covering map. 
\end{lemma}
\begin{proof}
We begin by proving that $\phi$ is in fact surjective. Let $v\in V$ be any vertex, and consider $\phi(v)\in V'$. Let $v'\in V'$. Since $\Ga'$ is connected there is a path 
$$\phi(v)=v_1\stackrel{e_1}{\dash} v_2\stackrel{e_2}{\dash}\cdots\stackrel{e_{n-1}}{\dash} v_n=v'.$$
We can lift this path step by step to a path $$v=w_1\stackrel{f_1}{\dash} w_2\stackrel{f_2}{\dash}\cdots\stackrel{f_{n-1}}{\dash} w_n$$
in $\Ga$ in the following way:
if $c(e_1)= (i,j)$ and $e_1$ is directed from $v_1$ to $v_2$, then we take $f_1$ to be the unique edge going out of $v=w_1$ with $c_o(f_1)=i$. Since $\phi$ is chromatic, the edge $\phi(f_1)$ is the unique edge with output color $i$ that starts in $\phi(w_1)=v_1$. It thus must be equal to $e_1$. We then define $w_2:=t(f_1)$. If $e_1$ is directed from $v_2$ to $v_1$ we use the input edge of $v$ with input color $j$. 
We continue by induction, and we get a vertex $w_n$ such that $\phi(w_n) = v_n=v'$. We thus see that $\phi$ is surjective on the set of vertices. It is now easy to see that it is also surjective on the set of edges, by considering the coloring. 

To prove that $\phi$ is also a covering is now immediate. Indeed, if $v\in V$ has color $a$ then $\phi(v)$ has color $a$ as well. Both $v$ and $\phi(v)$ have $q_a$ input edges and $p_a$ output edges, and since $\phi$ is a chromatic morphism it holds that it maps the edges adjacent to $v$ bijectively to the edges adjacent to $\phi(v)$, and $\phi$ is thus a local homeomorphism and therefore a covering. 
\end{proof}
\begin{lemma}\label{lem:uniquecoloring}
Let $\phi:\Ga\to \Ga'$ be a covering of graphs. Assume that $\Ga'$ is an adequate graph. Then there is a unique coloring on $\Ga$ that makes it into an adequate graph in such a way that $\phi$ is a chromatic morphism.
\end{lemma}
\begin{proof}
We define the color of $v\in V$ to be the color of $\phi(v)$ and the color of $e\in E$ to be the color of $\phi(e)$. Since $\phi$ is a local homeomoprhism this coloring really gives a structure of an adequate graph on $\Ga$: if $c(v)=a$ then $c(\phi(v))=a$ and therefore $\phi(v)$ has $p_a$ output edges with colors $1,\ldots, p_a$, and $q_a$ input edges with colors $1,\ldots, q_a$. Since $\phi$ is a local homeomorphism the same holds for $v$. Since a chromatic morphism preserves the coloring it is clear that this coloring is unique.
\end{proof}
\begin{lemma}\label{lem:uniquevertex}
Let $\phi_1,\phi_2:\Ga=(V,E)\to \Ga'=(V',E')$ be two chromatic morphisms between connected adequate graphs. Assume that there is a vertex $v\in V$ such that $\phi_1(v) = \phi_2(v)$. Then $\phi_1=\phi_2$. 
\end{lemma}
\begin{proof}
Let $v'\in V$. We will prove by induction on the length of the minimal path from $v$ to $v'$ that $\phi_1(v')=\phi_2(v')$. The fact that the images of the edges are also the same follows by considering the colors of the edges.

If $d(v,v')=0$ then $v=v'$ and we know that $\phi_1(v)=\phi_2(v)$. 
Assume that if $d(v,v')=n$ then $\phi_1(v')=\phi_2(v')$. 
If $v'$ is a vertex with $d(v,v')=n+1$ then there is a vertex $v''$ with $d(v,v'')=n$ and $d(v'',v')=1$. By the induction hypothesis we know that $\phi_1(v'')=\phi_2(v'')$. If $v''$ is connected to $v'$ by an edge $e$ with color $(i,j)$, then $\phi_1(v')$ is the unique vertex in $V'$ that is connected to $\phi_1(v'')$ by an edge with color $(i,j)$. Similarly, $\phi_2(v')$ is the unique vertex in $V'$ that is connected to $\phi_2(v'')=\phi_1(v'')$ by an edge with color $(i,j)$, and therefore $\phi_1(v')=\phi_2(v')$ as desired. 
\end{proof}
The next lemma is well known in graph theory. We recall it here:
\begin{lemma}
A graph $\Ga$ is connected and simply connected if and only if it is contractible if and only if it is a tree.
\end{lemma}
\begin{proof}
Any contractible topological space is connected and simply connected. By definition, a tree $T$ is a connected graph with no cycles, which is equivalent to $T$ being connected and simply connected. Finally, if $T$ is a tree then it can be contracted by picking a base point $t$ , and contracting every other point $s$ in $T$ to that point $t$ along the unique path between $s$ and $t$. See also Theorem 9.1. in \cite{Wilson} for more characterisations of trees.
\end{proof}
\begin{lemma}\label{lem:lifting}
Assume that $\phi:\Ga_1=(V_1,E_1)\to \Ga_2=(V_2,E_2)$ is a chromatic morphism between adequate graphs. Let $\wt{\Ga}=(\wt{V},\wt{E})$ be a connected simply connected adequate graph, and let $\psi:\wt{\Ga}\to \Ga_2$ be another chromatic morphism. Let $v_1\in V_1$, $v_2\in V_2$ $\wt{v}\in\wt{V}$ be vertices that satisfy $\phi(v_1)= v_2= \psi(\wt{v})$. Then there is a chromatic morphism $\wt{\psi}:\wt{\Ga}\to \Ga_1$ such that $\wt{\psi}(\wt{v}) = v_1$ and $\phi\wt{\psi}=\psi$. If $\Ga_1$ is also connected and simply connected, then $\wt{\psi}$ is in fact an isomorphism. 
\end{lemma}
\begin{proof}
This follows immediately from Proposition 1.33 in \cite{Hatcher}. By considering all the relevant colors we see that the resulting map $\wt{\psi}$ is indeed a chromatic morphism. If $\Ga_1$ is connected and simply connected then $\wt{\psi}$ is a covering map of trees, and therefore must be an isomorphism.
\end{proof}

\begin{definition}
A finite connected adequate graph $\Ga$ is called irreducible if any morphism $\phi:\Ga\to \Ga'$ of adequate graphs is an isomorphism.
\end{definition}
\begin{remark} In \cite{meirUIR} the term ``irreducible diagram'' was used differently. We changed the usage of the word here.\end{remark}
Let now $\Ga$ be a connected adequate graph. 
%The results that we will use here are taken from Section 1.3. and 1.A. in \cite{Hatcher}. 
Recall that the \emph{universal covering} of $\Ga$ is the unique tree $\wt{\Ga}$ for which there is a covering map $p:\wt{\Ga}\to \Ga$. 
By Lemma \ref{lem:uniquecoloring}, $\wt{\Ga}$ has a unique coloring that makes it an adequate graph. Fix a vertex $v\in \Ga$ and a vertex $\wt{v}\in p^{-1}(v)$. 

By Proposition 1.39 in \cite{Hatcher} the group $\pi_1(\Ga,v)$ is isomorphic to the group of deck transformations of $p:\wt{\Ga}\to \Ga$. 
By Proposition 1.40. in \cite{Hatcher} we also have an isomorphism $\wt{\Ga}/\pi_1(\Ga,v)\cong \Ga$. By considering the colorings of vertices and edges it is easy to see that every deck transformation is also a chromatic automorphism of the graph $\wt{\Ga}$. Write $\Aut_{C}(\wt{\Ga})$ for the group of such automorphisms.

The inclusion $\pi_1(\Ga,v)\subseteq \Aut_{C}(\wt{\Ga})$ then gives the following covering maps:
$$\wt{\Ga}\stackrel{p}{\to} \wt{\Ga}/\pi_1(\Ga,v)\cong \Ga\stackrel{q}{\to} \Ga_0:=\wt{\Ga}/\Aut_{C}(\wt{\Ga}).$$
We claim that $\Ga_0$ is an irreducible graph. Indeed, if we have a non-injective map $r:\Ga_0\to \Ga_1$ onto an adequate connected graph $\Ga_1$,  then since the map $r$ is a covering it follows that there is a vertex $w\in \Ga_0$ such that $w\neq q(v)$ and $rq(v)=r(w)$. By considering now the covering $rqp:\wt{\Ga}\to \Ga_1$, and using the fact that $\wt{\Ga}$ is simply connected, we see that if $\wt{w}\in (qp)^{-1}(w)$ then by Lemma \ref{lem:lifting} there is a deck transformation $\phi:\wt{\Ga}\to\wt{\Ga}$ with respect to the covering $rqp$ that takes $\wt{v}$ to $\wt{w}$. Such a deck transformation is in particular a chromatic automorphism of $\wt{\Ga}$. But this implies in particular that $\wt{v}$ and $\wt{w}$ are in the same $\Aut_{C}(\wt{\Ga})$-orbit of $\wt{\Ga}$, and therefore $q(v) = qp(\wt{v}) =qp(\wt{w}) = w$, which is a contradiction. 

We thus see that $\Ga_0$ is an irreducible graph. We also see that $\Ga_0$ is the only irreducible graph that $\wt{\Ga}$ covers. This is because if $\wt{\Ga}\to \Ga_2$ is another covering of an irreducible graph then by the same argument presented above we get $\wt{\Ga}/\Aut_{C}(\wt{\Ga})\cong \Ga_2$. Since this quotient is already isomorphic to $\Ga_0$ we get the uniqueness. 

The covering map $qp:\wt{\Ga}\to \Ga_0$ is also unique. If there is another map $p':\wt{\Ga}\to \Ga_0$, then by Lemma \ref{lem:uniquevertex} 
we see that $p'(\wt{v})\neq qp(\wt{v})$. Take an element $z\in (qp)^{-1}p'(\wt{v})$. Then by Lemma \ref{lem:lifting} there is an adequate graph automorphism $\phi:\wt{\Ga}\to\wt{\Ga}$ such that $\phi(\wt{v})=z$. This contradicts the fact that $\wt{v}$ and $z$ are in different $\Aut_{C}(\wt{\Ga})$-orbits. 

We summarize this in the following proposition:
\begin{proposition}\label{prop:uniquecover}
Let $\Ga$ be a finite connected adequate graph. Then there is a unique (up to isomorphism) irreducible graph $\Ga_0$ for which there is a chromatic covering $q:\Ga\to\Ga_0$. The covering map is unique, and an automorphism $\phi:\Ga\to\Ga$ preserves the coloring of $\Ga$ if and only if it is a deck transformation.
\end{proposition}
\begin{proof}
We need to prove the uniqueness of the map $q:\Ga\to\Ga_0$. If $q':\Ga\to\Ga_0$ is another covering map then $pq,pq':\wt{\Ga}\to \Ga_0$ are two different coverings, and the discussion above shows that this is impossible. If $\phi:\Ga\to \Ga$ is a deck transformation, then for every $v\in V$ it holds that $c(\phi(v)) = c(q(\phi(v))) = c(q(v)) = c(v)$, so $\phi$ preserves the coloring of the vertices, and by a similar argument it also preserves the colors of the edges. In the other direction, if $\phi:\Ga\to \Ga$ is a chromatic automorphism, then $q\phi:\Ga\to \Ga_0$ is a chromatic covering. By the uniqueness of $q$ we get that $q\phi = q$, and therefore $\phi$ is a deck transformation. 
\end{proof}
Let $\Ga_0$ be an irreducible graph with a basepoint $v\in \Ga_0$. Choose a maximal tree $T$ in $\Ga_0$. This means that $T$ is maximal connected and acyclic. Since $\Ga_0$ is a graph, $\pi_1(\Ga_0,v)$ is a free group. If $e_1,\ldots, e_m$ are all the edges of $\Ga_0$ outside $T$, then $\pi_1(\Ga_0,v)$ is free of rank $m$ and we write $\pi_1(\Ga_0,v)\cong F_m = \langle z_1,\ldots, z_m\rangle$, where each edge gives rise to a generator of $F_m$. 
\begin{proposition}\label{prop:bijection}
Let $\Ga_0$ be as above, and let $n\geq 0$.
We have bijections between the following sets: 
\begin{enumerate}
\item The set of isomorphism classes of $n$-fold coverings of $\Ga_0$ 
\item The set of isomorphism classes of $\pi_1(\Ga_0,v)$-sets of cardinality $n$.
\item The set of conjugacy classes of homomorphisms $\pi_1(\Ga_0,v)\to S_n$.
\item The set $(S_n^m)/\sim$, where $(\sigma_1,\ldots, \sigma_m)\sim(\tau_1,\ldots, \tau_m)$ if and only if there is a permutation
$\mu\in S_n$ such that $\forall i\;\mu\sigma_i\mu^{-1} = \tau_i$
\end{enumerate}
\end{proposition}
\begin{proof}
The bijection between the first two sets follows from the discussion in pages 68-70 in \cite{Hatcher}. 
We sketch it here. If $S$ is a finite $F_m$-set with $|S|=n$, then $(S\times \wt{\Ga_0})_{F_m}$, where $F_m$ acts diagonally on the product, is an $n$-fold covering of $\Ga_0$.
If $p:\Ga\to \Ga_0$ is an $n$-fold covering of $\Ga_0$, write $p^{-1}(v) = \{v_1,\ldots, v_n\}$. For every element $g\in F_m$ and every $v_i\in p^{-1}(v)$ we can lift $g$ to a path $\wt{g}$ in $\Ga$ that begins with $v_i$. The end point of $\wt{g}$ is then $g\cdot v_i$. 

The bijection between the second and third sets holds for general groups, not only $F_m$. It is given by choosing a bijection between the $n$-elements set on which $\pi_1(\Ga_0,v)$ acts and the set $\{1,\ldots, n\}$. 

If $\alpha:\pi_1(\Ga_0,v)= \langle z_1,\ldots, z_m\rangle\to S_n$ is a homomorphism, then $(\alpha(z_1),\ldots,\alpha(z_m))$ is an element in $S_n^m$. 
It is then easy to see that this correspondence induces the bijection between the third set and the fourth set. 
\end{proof}

By the above proof we see that if $p:\Ga\to \Ga_0$ is a covering that corresponds to the tuple $(\sigma_1,\ldots \sigma_n)$ then we can write $p^{-1}(v) = \{v_1,\ldots, v_n\}$, and for $i=1,\ldots, m$ it holds that the different liftings of the path $z_i$ from $\Ga_0$ to $\Ga$ connect $v_j$ to $v_{\sigma_i(j)}$. 

We claim now the following:
\begin{lemma}
Let $\Ga_0$ be an irreducible graph and let $p:\Ga\to \Ga_0$ be an $n$-fold covering that corresponds to a tuple $(\sigma_1,\ldots, \sigma_m)\in S_n^m$, by identifying $p^{-1}(v)\cong \{1,\ldots, n\}$. Then the group $\Aut_{C}(\Ga)$ is isomorphic to the group $\{\sigma| \forall i \hspace{0.2cm} \sigma\sigma_i = \sigma_i\sigma\}<S_n$. This group can also be identified the with group $\Aut_{F_m}(p^{-1}(v))$. 
\end{lemma}
\begin{proof}
We have already seen in Proposition \ref{prop:uniquecover} that chromatic automorphisms of $\Ga$ are the same as deck transformations of $\Ga$. Let $\phi:\Ga\to \Ga$ be such an automorphism. Write as before $p^{-1}(v) = \{v_1,\ldots, v_n\}$. Since $\phi$ is a deck transformation it permutes the set $p^{-1}(v)$. Write $\alpha(\phi)\in S_n$ for the unique permutation that satisfies $\phi(v_i) = v_{\alpha(\phi)(i)}$. Then $\alpha:\Aut_{C}(\Ga)\to S_n$ is a homomorphism. Moreover, since a deck transformation is determined by its value on a single vertex (Lemma \ref{lem:uniquevertex}), $\alpha$ is one to one.

Write now $\alpha(\phi) = \sigma$. Let $1\leq 1\leq m$ and let $1\leq j\leq n$. Write $z_i^j$ for the path in $\Ga$ that lifts $z_i$ and starts in $v_j$. The end point of $z_i^j$ is then $v_{\sigma_i(j)}$. Then $\phi(z_i^j)$ is the path that starts in $v_{\sigma(j)}$ and ends in $v_{\sigma\sigma_i(j)}$. But a lifting of $z_i$ that starts in $v_{\sigma(j)}$ must end in $v_{\sigma_i\sigma(j)}$. By uniqueness of the end points of liftings we get that $\sigma \sigma_i = \sigma_i\sigma$ as desired.  Since $F_m$ acts on $p^{-1}(v)$ via the permutations $\sigma_i$, we see that we can indeed identify the group $\{\sigma| \forall i \hspace{0.2cm} \sigma\sigma_i = \sigma_i\sigma\}<S_n$ with $\Aut_{F_m}(p^{-1}(v))$. If $\sigma\in S_n$ is in this group, then $\sigma\times Id: p^{-1}(v)\times \wt{\Ga_0}\to p^{-1}(v)\times \wt{\Ga_0}$ commutes with the $F_m$-action and thus induces a homeomorphism $\ol{\sigma\times Id}:(p^{-1}(v)\times \wt{\Ga_0})_{F_m}\to (p^{-1}(v)\times \wt{\Ga_0})_{F_m}$, where for a graph $R$ with an $F_m$-action, $R_{F_m}$ denotes the quotient of $R$ by the action of $F_m$.  
By the discussion above, this covering space is isomorphic to $\Ga$, and we thus see that the image of $\alpha$ is $\Aut_{F_m}(p^{-1}(v))$ indeed. 
\end{proof}

We summarize this section with the following proposition:
\begin{proposition}\label{prop:splitting1}
\begin{enumerate}
\item The algebra $\Ainf$ splits as $$\Ainf = \bigotimes_{\Ga_0}\Ainf_{\Ga_0},$$ where $\Ga_0$ runs over all irreducible graphs and $\Ainf_{\Ga_0}$ is the polynomial algebra generated by all the connected graphs that cover $\Ga_0$.
\item The algebra $\Ainf_{\Ga_0}$ is graded by $\N$, and $(\Ainf_{\Ga_0})_n\cong (KS_n^m)_{S_n}$, where $m$ is the rank of the fundamental group of $\Ga_0$, and $S_n$ acts on $S_n^m$ by diagonal conjugation. If $\Ga_0$ has degree $(n_1,\ldots, n_r)$ as an element of $\Ainf$, then the $n$-th homogeneous component of $\Ainf_{\Ga_0}$ has degree $(nn_1,\ldots, nn_r)$ in $\Ainf$. 
\item If $\Ga\in \Ainf_{\Ga_0}$ corresponds to the tuple $\ol{(\sigma_1,\ldots, \sigma_m)}$, then $||\Ga||^2 = |\{\sigma\in S_n| \forall i\hspace{0.2cm} \sigma\sigma_i = \sigma_i\sigma\}$. 
\item The multiplication in $\Ainf_{\Ga_0}$ is given by the rule $(\sigma_i)\cdot (\tau_i) = ((\sigma_i,\tau_i))$, where we identify $S_{n_1}\times S_{n_2}$ with $S_{n_1,n_2}\subseteq S_{n_1+n_2}$.

\end{enumerate}
\end{proposition}
\begin{proof}
The first part follows from the fact that every connected graph covers a unique irreducible graph, and every graph can be written uniquely as the disjoint union of its connected components. The second part follows from the above discussion. The third part follows from the fact that the squared norm of every diagram is the cardinality of its automorphism groups (see also \cite[Section 8]{meirUIR}). The formula for the multiplication follows by considering the correspondence between $F_m$-sets and coverings. 
\end{proof}

\begin{remark}
\begin{enumerate}
\item If $m=0$ then $\Ga_0$ is a tree, the only covering spaces of $\Ga_0$ is itself. As a result, $\Ainf_{\Ga_0}$ is a polynomial ring in one variable.

\item Isomorphism classes of finite transitive $F_m$-sets are classified by conjugacy classes of finite index subgroups of $F_m$, and every finite $F_m$-set can be written uniquely as a disjoint union of finite transitive $F_m$-sets. Most of the results in this paper are easier to describe by just using the general language of $F_m$-sets, and not referring to particular subgroups. 
\end{enumerate}
\end{remark}

Below is an example of a diagram with an associated irreducible graph, and  of 2-fold and 3-fold coverings of the diagram and the associated graphs:
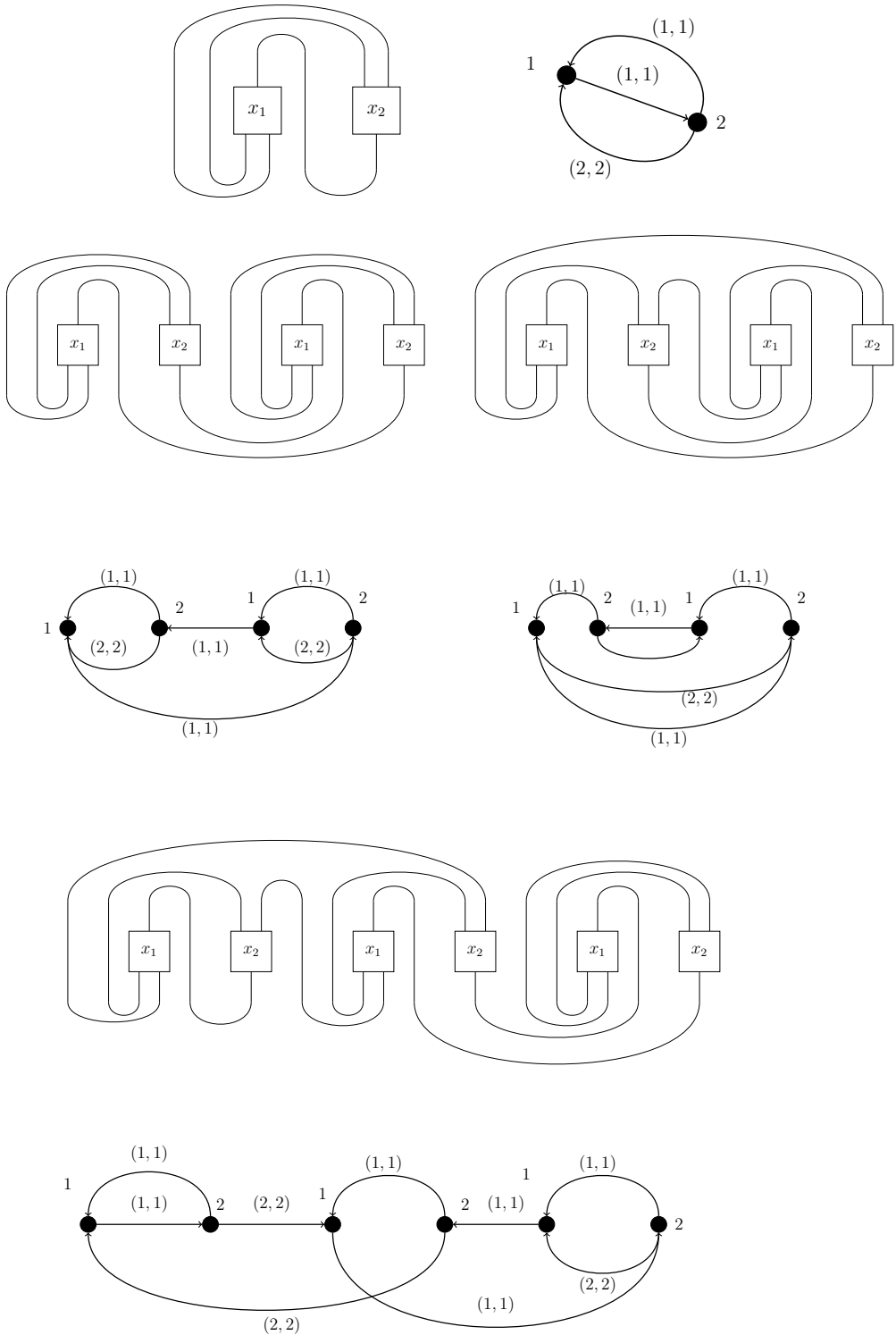
\begin{center}\scalebox{0.7}{
\begin{tikzpicture}
	\begin{pgfonlayer}{nodelayer}
		\node [style=2function] (0) at (-3.25, 4.5) {$x_1$};
		\node [style=none] (1) at (-3.25, 4.75) {};
		\node [style=none] (2) at (-3, 4.25) {};
		\node [style=none] (3) at (-3.5, 4.25) {};
		\node [style=none] (4) at (-3.25, 5.75) {};
		\node [style=none] (5) at (-3, 3.25) {};
		\node [style=none] (6) at (-3.5, 3.25) {};
		\node [style=2function] (7) at (-0.75, 4.5) {$x_2$};
		\node [style=none] (8) at (-0.75, 3.25) {};
		\node [style=none] (9) at (-0.5, 5.75) {};
		\node [style=none] (10) at (-1, 5.75) {};
		\node [style=none] (11) at (-0.75, 4.25) {};
		\node [style=none] (12) at (-0.5, 4.75) {};
		\node [style=none] (13) at (-1, 4.75) {};
		\node [style=none] (14) at (-2.25, 5.75) {};
		\node [style=none] (15) at (-2.25, 3.25) {};
		\node [style=none] (16) at (-4.25, 5.75) {};
		\node [style=none] (17) at (-4.25, 3.25) {};
		\node [style=none] (18) at (-5, 5.75) {};
		\node [style=none] (19) at (-5, 3.25) {};
		\node [style=none] (22) at (4.75, 5.25) {$(1,1)$};
		\node [style=none] (23) at (5.5, 6.25) {$(1,1)$};
		\node [style=none] (24) at (3.75, 3.25) {$(2,2)$};
		\node [style=none] (25) at (2.5, 5.5) {1};
		\node [style=none] (26) at (6.5, 4.25) {2};
		\node [style=graph node] (27) at (3.25, 5.25) {};
		\node [style=graph node] (28) at (6, 4.25) {};
	\end{pgfonlayer}
	\begin{pgfonlayer}{edgelayer}
		\draw (3.center) to (6.center);
		\draw (2.center) to (5.center);
		\draw (4.center) to (1.center);
		\draw (10.center) to (13.center);
		\draw (9.center) to (12.center);
		\draw (11.center) to (8.center);
		\draw [bend left=90, looseness=1.25] (4.center) to (14.center);
		\draw (14.center) to (15.center);
		\draw [bend right=90, looseness=1.25] (15.center) to (8.center);
		\draw (16.center) to (17.center);
		\draw (18.center) to (19.center);
		\draw [bend left=90, looseness=0.75] (18.center) to (9.center);
		\draw [bend left=90, looseness=0.75] (16.center) to (10.center);
		\draw [bend right=90, looseness=1.50] (17.center) to (6.center);
		\draw [bend right=90] (19.center) to (5.center);
		\draw [style=new edge style 0, bend right=270, looseness=1.25] (28) to (27);
		\draw [style=new edge style 0, bend right=90, looseness=1.25] (28) to (27);
		\draw [style=new edge style 0] (27) to (28);
	\end{pgfonlayer}
\end{tikzpicture}}
\end{center}
\begin{center}\scalebox{0.6}{
\input{diagram2.tikz}}
\end{center}

\section{decompositions of group algebras}\label{sec:decomposing}
Let $G$ be a finite group, and let $H$ be a subgroup of $G$. In this section we will study the space of coinvariants $(KG)_H$, where $H$ acts on $KG$ by conjugation. The motivation for this is the fact that the $n$-th homogeneous component of $\Ainf_{\Ga_0}$ is isomorphic with $(KS_n^m)_{S_n}$, where $S_n$ is embedded in $S_n^m$ diagonally. The space $(KG)_H$ is equipped with the sesquilinear inner product given by 
$$\langle \ol{g_1},
\ol{g_2}\rangle = \begin{cases} |C_H(g_1)| & \text{ if } \ol{g_1}=\ol{g_2}\\ 0 & \text{ else} \end{cases},$$ where $\ol{g}$ denotes the $H$-conjugacy class $\{hgh^{-1}\}_{h\in H}$ and $C_H(g) = \{a\in H| aga^{-1}=g\}$.
If $H=G$, then this space is spanned by the images of the central idempotents in $KG$, which correspond to the irreducible representations of $G$. By rescaling by the dimensions of the irreducible representations we get an orthonormal basis. If $H=1$, then this space has an orthonormal  basis given by the group elements of $G$. In the general case we have a basis that is given by a mixture of group theory and representation theory, as we shall describe next. 

Write  $G= \sqcup_{x\in D} HxH$, where $D$ is a set of double-coset representatives. We can then write 
$$(KG)_H = \bigoplus_{x\in D} (KHxH)_H.$$ Up to $H$-conjugation, every element in $HxH$ is conjugate to an element of the form $Hx$. It holds that $h_1x$ and $h_2x$ are $H$-conjugate if and only if there is an element $h_3\in H\cap x^{-1}Hx$ such that $h_1 = h_3h_2xh_3^{-1}x^{-1}$. We thus see that $$(KHxH)_H \cong (KH)_{H\cap x^{-1}Hx},$$ where the action of $H\cap x^{-1}Hx$ on $H$ is given by $h'\cdot h = h'hxh'^{-1}x^{-1}$. We can decompose now $H$ again to double cosets and continue with this procedure inductively. This motivates the following definition:
\begin{definition}
Let $G$ be a finite group, and let $H$ be a subgroup. For $x\in X$ we define $H_x: = \bigcap_{i\in \Z} x^iHx^{-i}$. We write $S_x = H_x\cdot x\subseteq H$. 
\end{definition}
The subgroup $H_x$ is the biggest subgroup of $H$ that is normalized by $x$. 
We claim the following:
\begin{lemma} If $y\in S_x$ then $H_y=H_x$ and as a result $S_x=S_y$. For $x,y\in G$ it thus holds that $S_x=S_y$ or $S_x\cap S_y = \varnothing$. 
\end{lemma}
\begin{proof}
If $y\in S_x$ then $y=hx$ where $h\in H_x$. Since $x$ normalizes $H_x$ we can write, for $n\geq 0$, $$y^n =(hx)^n = x^n(x^{-n}hx^n)\cdot (x^{1-n}hx^{n-1})\cdots (x^{-1}hx) = x^n h_n$$ for some $h_n\in H_x$. We can show that a similar result holds when $n<0$. It then holds that
$$H_y = \bigcap_{n\in \Z} y^n Hy^{-n} = \bigcap_{n\in \Z} x^nh_nHh_n^{-1}x^{-n} = \bigcap_{n\in \Z} x^nHx^{-n} = H_x.$$
We then also have $S_y = H_yy = H_xhx = H_xx = S_x$ as desired. 
For the second part, if $S_x\cap S_y\neq \varnothing$, then we can take $z\in S_x\cap S_y$, and then $S_x=S_z=S_y$. 
\end{proof} 
Thus, the different subsets $S_x$ partition the group $G$ into mutually disjoint subsets.
For $h\in H$ and $x\in G$ it holds that $$hH_xh^{-1} = \bigcap_{n\in \Z}hx^nHx^{-n}h^{-1} = \bigcap_{n\in \Z}hx^nh^{-1}Hhx^{-n}h^{-1} = \bigcap_{n\in \Z}(hxh^{-1})^nH(hxh^{-1})^{-n} = H_{hxh^{-1}}$$ and therefore 
$$S_{hxh^{-1}} = H_{hxh^{-1}}hxh^{-1} = hH_xh^{-1}hxh^{-1} = hH_xxh^{-1} = hS_xh^{-1}$$
\begin{definition} We define $x\sim y$ if there is an $h\in H$ such that $hS_xh^{-1} = S_y$.  
\end{definition}
Since $\{S_x\}$ is a partition of $G$, we see that $\sim$ is an equivalence relation, where the equivalence class of $x\in G$ is $$T_x:=\bigcup_{h\in H} S_{hxh^{-1}}.$$ We next claim the following:
\begin{lemma}
Assume that $hS_xh^{-1} = S_x$ for some $x\in G$ and $h\in H$. Then $h\in H_x$. In particular, if $h\in H$ satisfies that $hS_xh^{-1}\cap S_x\neq \varnothing$ then $h\in H_x$. 
\end{lemma}
\begin{proof}
Since $x\in S_x$ we get that $hxh^{-1} = ax$ for some $a\in H_x$. This gives $h^{-1}a = xh^{-1}x^{-1}$. We conjugate the last equation by $x^n$ where $n\in \Z$ and we get $$x^nh^{-1}x^{-n}\cdot x^n a x^{-n}  = x^{n+1}h^{-1}x^{-n-1}.$$ Since $x^n a x^{-n}\in H$ for every $n$, it holds that $x^nhx^{-n}\in H$ if and only if $x^{n+1}hx^{-n-1}\in H$. Since this is true for $n=0$, it is also true for all $n\in \Z$ and thus $h\in H_x$. 
The last statement follows from the fact that $hS_xh^{-1} = S_{hxh^{-1}}$, and thus if the intersection is not empty then the two sets must be equal. 
\end{proof}
The group $H_x$ acts on $S_x$ by conjugation. We have 
$$(KS_x)_{H_x} \cong (KH_x)_{H_x},$$ where the action of $H_x$ on $KH_x$ is given by $a\cdot b = abxax^{-1}$. Pick now a set $\{g_1,\ldots, g_r\}$ of equivalence class representatives for $\sim$. Since the equivalence classes for $\sim$ are closed under $H$-conjugation, we get
$$(KG)_H = \bigoplus_i (KT_{g_i})_H.$$
Every element in $T_{g_i}$ is conjugate to an element in $S_{g_i}$, and we just proved that two elements in $S_{g_i}$ are $H$-conjugate if and only if they are $H_{g_i}$-conjugate. It follows that $(KT_{g_i})_H\cong (KS_{g_i})_{H_{g_i}}$. 

We thus have 
\begin{equation}\label{eq:splittingKG}(KG)_H\cong \bigoplus_i (KH_{g_i})_{H_{g_i}},\end{equation}
where $H_{g_i}$ acts on $KH_{g_i}$ by the formula $a\cdot h = ahg_iag_i^{-1}$. We will call this action the \emph{$g_i$-twisted conjugation action}. The direct sum respects the inner product in the sense that the different direct summands are orthogonal to each other.  
We calculate the inner product on $(KH_{g_i})_{H_{g_i}}$. Following all the isomorphisms we have so far we get
\begin{equation}\label{eq:sqnorm}\langle \ol{h_1},\ol{h_2}\rangle = \begin{cases}  |C_H(h_1g_i)| & \text{ if } \ol{h_1}=\ol{h_2}\\ 0 & \text{ else}\end{cases} \end{equation}
If $a\in C_H(h_1g_i)$ then $aS_{g_i}a^{-1} = aS_{h_1g_i}a^{-1} = S_{ah_1g_ia^{-1}} = S_{h_1g_i}$ and by the above lemma we get that in fact $a\in H_{g_i}$, so $C_H(h_1g_i) = C_{H_{g_i}}(h_1g_i)$.

We have a natural isomorphism $$(KH_{g_i})_{H_{g_i}}\cong (KH_{g_i})^{H_{g_i}}$$
$$\ol{h}\mapsto\frac{1}{|H_{g_i}|}\sum_{a\in H_{g_i}} ahg_ia^{-1}g_i^{-1}$$
We can thus identify $(KH_{g_i})_{H_{g_i}}$ with $(KH_{g_i})^{H_{g_i}}$, which is in turn a subspace of $KH_{g_i}$. 
We claim that the inner product on $(KH_{g_i})_{H_{g_i}}$ is just the restriction of the sesquilinear inner product on $KH_{g_i}$ given by
$$\langle h_1, h_2\rangle = \begin{cases} |H_{g_i}| & \text{ if } h_1=h_2 \\ 0 & \text{ else } \end{cases} .$$
Indeed, since the isomorphism between the invariants and the coinvariants sends an orthogonal basis to an orthogonal basis, it is enough to check that the two inner products agree on the squared norms of elements in the bases. Write $q_1,\ldots, q_b$ for a set of coset representatives of $C_H(hg_i)$ in $H_{g_i}$. We then have 
$$||\frac{1}{|H_{g_i}|}\sum_{a\in H_{g_i}} ahg_ia^{-1}g_i^{-1}||^2 = \frac{1}{|H_{g_i}^2|}||\sum_{j=1}^b |C_{H_{g_i}}(hg_i)|q_jhg_iq_j^{-1}g_i^{-1}||^2 =$$ $$ \frac{|C_H(hg_i)|^2}{|H_{g_i}^2|}b |H_{g_i}| = |C_{H_{g_i}}(hg_i)|,$$ where we used the fact that $b= |H_{g_i} / C_{H_{g_i}}(hg_i)|$. 

Write now $\{W^{(i)}_j\}$ for the set of isomorphism classes of irreducible representations of $H_{g_i}$. 
Wedderburn-Artin Theorem enables us to write $KH_{g_i} \cong \bigoplus_j \Hom_K(W^{(i)}_j,W^{(i)}_j)$.
The isomorphism sends $h\in H_{g_i}$ to the tuple $(R_j)$, where $R_j:W_j^{(i)}\to W_j^{(i)}$ is given by the action of $h$ on the representation $W_j^{(i)}$. The $g_i$-twisted conjugation action preserves the direct sum decomposition, and the $g_i$-twisted conjugation action of $h\in H_{g_i}$ on $R:W_j^{(i)}\to W_j^{(i)}$ is the map $hRg_ih^{-1}g_i^{-1}:W_j^{(i)}\to W_j^{(i)}$.
We thus see that $R:W_j^{(i)}\to W_j^{(i)}$ is invariant under the $g_i$-twisted conjugation action if and only if it is $H_{g_i}$-equivariant when considered as a map $g_i^*(W_j^{(i)})\to W_j^{(i)}$. 
Since the invariants and coivariants are isomorphic, we get $$(KH_{g_i})_{H_{g_i}}\cong (KH_{g_i})^{H_{g_i}}\cong \bigoplus_j \Hom_{H_{g_i}}(g_i^*(W^{(i)}_j),W^{(i)}_j).$$

The representation $(g_i^*)(W^{(i)}_j)$ is again an irreducible representation. If $(g_i)^*(W^{(i)}_j)\cong W^{(i)}_j$ then $\Hom_{H_{g_i}}((g_i)^*(W^{(i)}_j),W^{(i)}_j)$ is one dimensional. Otherwise it is zero dimensional. 
This implies the following lemma:
\begin{lemma}
The dimension of $(KH_{g_i})_{H_{g_i}}$ is equal to the number of isomorphism classes of irreducible $H_{g_i}$ representations that are invariant under the action by conjugation of $g_i$. 
\end{lemma}
We claim the following:
\begin{lemma}
An irreducible representation $W^{(i)}_j$ is invariant under the action of $(g_i)^*$ if and only if it can be extended to a representation of $H_{g_i}\cdot \langle g_i\rangle$. 
\end{lemma}
\begin{proof}
One direction is obvious, since if $W=W^{(i)}_j$ can be extended to a representation of $H_{g_i}\cdot \langle g_i\rangle$, then the action of $g_i$ on $W$ gives an isomorphism between $W$ and $(g_i)^*(W)$.
In the other direction, assume that $k$ is the minimal integer such that $g_i^k=a\in H_{g_i}$.  Assume that $(g_i)^*(W)\cong W$, and let $T:W\to W$ be a linear automorphism that satisfies $T(hw) = g_ihg_i^{-1}T(w)$ for $h\in H_{g_i}$. It holds that $T^k(hw) = g_i^khg_i^{-k}T^k(w) = aha^{-1}T^k(w)$. It follows that $a^{-1}T^k(hw) = ha^{-1}T^k(w)$. So $w\mapsto a^{-1}T^k(w)$ is $H_{g_i}$-equivariant. By Schur's Lemma this means that it is multiplication by some scalar. By rescaling $T$ and using the fact that $K$ is algebraically closed, we can assume that this scalar is 1. We then get a representation of $H_{g_i}\cdot \langle g_i\rangle$ on $W$, where $h\cdot g_i^l$ acts by $h\cdot T^l$. The fact that $k$ is the minimal integer such that $g_i^k\in H_{g_i}$ and that $g_i$ normalizes $H_{g_i}$ implies that we can write every element of $H_{g_i}\cdot\langle g_i\rangle$ uniquely as $h\cdot g_i^l$ for some $l\in \{0,1,\ldots, k-1\}$, so we are done. 
\end{proof}
\begin{remark} In the above proof we have $k$ different extensions of $W$ to a representation of $H_{g_i}\cdot \langle g_i\rangle$, as we can alter $T$ by a $k$-th root of unity.
\end{remark}
The last lemma provides us with an orthonormal basis for $(KH_{g_i})_{H_{g_i}}$. By classical representation theory we know that the sesquilinear inner product on $H_{g_i}$ is in fact equal to $\langle h_1,h_2\rangle = \chi_{reg}(h_1h_2^{-1})$, where $\chi_{reg}$ is the character of the regular representation. The restriction to $\End(W^{(i)}_j)$ is $\langle T_1,T_2\rangle=\dim(W^{(i)}_j)\cdot \Tr(T_1\cdot T_2^*)$, where $T_2^*$ is the adjoint of $T_2$ (we use the fact that every irreducible representation admits an invariant sesquilinear form).
Assume that $W^{(i)}_j$ is $(g_i)^*$-invariant. Choose an extension of $W^{(i)}_j$ to a representation of $H_{g_i}\cdot\langle g_i\rangle$. Then $T^{(i)}_j:=g_i^{-1}:W^{(i)}_j\to W^{(i)}_j$ is an element of $\Hom_{H_{g_i}}((g_i)^*(W^{(i)}_j),W^{(i)}_j)$. By considering the character of the regular representation in the bigger group $H_{g_i}\cdot \langle g_i\rangle$, we see that $\langle T^{(i)}_j,T^{(i)}_j\rangle = \dim(W^{(i)}_j)\Tr(g_i^{-1}g_i|_{W^{(i)}_j})=\dim(W^{(i)}_j)^2$. 
We summarize this in the following:
\begin{lemma}
The set $\{\frac{1}{\dim(W^{(i)}_j)}T^{(i)}_j\}_{[W^{(i)}_j]\in \Irr(H_{g_i})^{g_i}}$ is an orthonormal basis for $(KH_{g_i})^{H_{g_i}}$, and the set $\{\frac{1}{\dim(W^{(i)}_j)}\ol{T^{(i)}_j}\}_{[W^{(i)}_j]\in \Irr(H_{g_i})^{g_i}}$ is an orthonormal basis for $(KH_{g_i})_{H_{g_i}}$, 
\end{lemma}
%\begin{definition} We will denote by $\{v^i_j\}\subseteq (KH_{g_i})_{H_{g_i}}$ the image of the above orthonormal basis under the isomorphism $(KH_{g_i})^{H_{g_i}}\cong (KH_{g_i})_{H_{g_i}}$. Since the isomorphism preserves the inner product, $\{v^i_j\}$ is an orthonormal basis for $(KH_{g_i})_{H_{g_i}}$. \end{definition}
From now on, whenever we have an automorphism $\nu:G\to G$ of a finite group, and an isomorphism $T:\nu^*(W)\to W$, we will assume that $T$ has finite order when considered as a linear map $W\to W$. 

We finish this section with a lemma that will be used later when calculating the multiplication explicitly in terms of the basis we describe here. 
To state the lemma, let $G_1\subseteq G_2$ be finite groups, and let $\nu:G_2\to G_2$ be an automorphism such that $\nu(G_1)=G_1$. Let $V$ be an irreducible $\nu$-invariant representation of $G_1$, let $W$ be an irreducible $\nu$-invariant representation of $G_2$, and let $T_V:\nu^*(V)\to V$ and $T_W:\nu^*(W)\to W$ be isomorphisms. 
\begin{lemma}\label{lem:innerprod}
Write $\beta:(KG_1)_{G_1}\to (KG_2)_{G_2}$ for the map $\ol{g}\mapsto \ol{g}$, where the action of $G_i$ on $KG_i$ is defined as $g\cdot h = gh\nu(g)^{-1}$ for $i=1,2$. Write $\gamma:KG_1\to KG_2$ for the natural inclusion. Then $\langle \beta(\frac{1}{\dim(V)}\ol{T_V}),\frac{1}{\dim(W)}\ol{T_W} \rangle = \frac{1}{\dim(V)}\chi_W(\gamma(T_V)T_W^*)$
\end{lemma}
\begin{proof}
The map $\beta$ is induced by the inclusion $\gamma$. The inner product on $(KG_2)_{G_2}$ is given by $\chi_{reg}(-,(-)^*)$, so we get 
$\langle \beta(\frac{1}{\dim(V)}\ol{T_V}),\frac{1}{\dim(W)}\ol{T_W} \rangle = \chi_{reg}((\frac{1}{\dim(V)\dim(W)}\gamma(T_V)T_W^*)$. Since the restriction of $\chi_{reg}$ to $\End(W)$ is given by $\dim(W)\chi_W$ we have the result. 
\end{proof}

\section{The case $G=S_n^m, H=S_n$}
Let $\Ga_0$ be an irreducible graph with $\pi_1(\Ga_0,v_0) = F_m$. 
We consider now the space $(KG)_H$ in case $G=S_n^m$ and $H=S_n$, embedded diagonally in $G$, following Proposition \ref{prop:bijection}. 
We have already seen that elements of $G$ can be thought of as $F_m$-sets of cardinality $n$, and two elements in $G$ are $H$-conjugate if and only if they define isomorphic $F_m$-sets. We next determine the equivalence relation $\sim$ in this case.
For this, define $\upsilon:F_m\to \Z$ to be the group homomorphism given on the generators of $F_m$ by $\upsilon(z_i)=1$ for every $i\in\{1,\ldots,m\}$. Write $F_m^0 = \Ker(\upsilon)$. The group $F_m$ splits as a semidirect product $\langle z_1\rangle \ltimes F_m^0$. 
The following lemma is easy to prove:
\begin{lemma}
If $m=1$ then $F_m^0=1$. If $m>1$ then the group $F_m^0$ is freely generated by the elements $z_j^{-a}z_1^a$ for $a\in\Z$, $j\in\{2,\ldots, m\}$ .
\end{lemma}
A tuple $(\sigma_i)\in S_n^m$ defines an action of $F_m$ on $\{1,\ldots, n\}$ and by restriction also an action of $F_m^0$ on the same set. We denote this set with the $F_m$-action by $Q_{(\sigma_i)}$. We write $\phi_{(\sigma_i)}:F_m\to S_n$ for the group homomorphism that sends $z_i$ to $\sigma_i$. 
We claim the following:
\begin{lemma}\label{lem:similarity}
Let $g=(\sigma_i)\in S_n^m=G$, and let $H=S_n$ with the diagonal embedding in $S_n^m$. 
\begin{enumerate}
\item We have $H_g = \Aut_{F_m^0}(Q_{(\sigma_i)})= C_{S_n}(\phi_{(\sigma_i)}(F_m^0))$. 
\item We have $(\sigma_i)\sim(\tau_i)$ if and only of $Q_{(\sigma_i)}\cong Q_{(\tau_i)}$ as $F_m^0$-sets.
\end{enumerate}
\end{lemma}
\begin{proof}
We calculate $H_g$. For $j\in \Z$ we have 
$$g^jHg^{-j} = \{(\sigma_i^j\sigma\sigma_i^{-j})|\sigma\in S_n\}.$$
Therefore, an element $(\tau,\tau,\ldots,\tau)$ is in $g^jHg^{-j}\cap H$ if and only if $$\forall i,k\in \{1,\ldots, m\}\, \sigma_i^j\tau\sigma_i^{-j} = \sigma_k^j\tau\sigma_k^{-j}.$$ In other words, $\tau\in g^jHg^{-j}\cap H$ if and only if $\tau$ commutes with all the elements of the form $\sigma_k^{-j}\sigma_i^j$ for every $i,j,k$. Since the elements $z_k^{-j}z_i^j$ generate $F_m^0$ we get the first assertion.

For the second claim, assume first that $Q_{(\sigma_i)}\cong Q_{(\tau_i)}$ as $F_m^0$-sets. Without loss of generality we can conjugate with an element in $S_n$ and assume that $Q_{(\sigma_i)}  = Q_{(\tau_i)}$, since this will not change the $\sim$-equivalence class. 
We thus know that $\phi_{(\sigma_i)}|_{F_m^0} = \phi_{(\tau_i)}|_{F_m^0}$. Write $\mu=\tau_1^{-1}\sigma_1$. 
If $f\in F_m^0$ then write $\beta = \phi_{(\sigma_i)}(f) = \phi_{(\tau_i)}(f)$. We have $$\mu\beta\mu^{-1} = \phi_{(\tau_i)}(z_1)^{-1}\phi_{(\sigma_i)}(z_1)\phi_{(\sigma_i)}(f)\phi_{(\sigma_i)}(z_1^{-1})\phi_{(\tau_i)}(z_1) = $$
$$\phi_{(\tau_i)}(z_1)^{-1}\phi_{(\sigma_i)}(z_1fz_1^{-1})\phi_{(\tau_i)}(z_1) = $$
$$\phi_{(\tau_i)}(z_1)^{-1}\phi_{(\tau_i)}(z_1fz_1^{-1})\phi_{(\tau_i)}(z_1) = $$
$$\phi_{(\tau_i)}(g) = \beta.$$
So $\mu$ commutes with the image of $F_m^0$ under $\phi_{(\tau_i)}$. We now have $\tau_i^{-1}\sigma_i = \tau_i^{-1}\tau_1\tau_1^{-1}\sigma_1\sigma_1^{-1}\sigma_i = \sigma_i^{-1}\sigma_1 \mu\sigma_1^{-1}\sigma_i = \mu$, where we have used the fact that the restrictions of the two homomorphisms to $F_m^0$ are the same, and the fact that $\mu$ commutes with the image of $F_m^0$ under the image of the restrictions. This implies that $(\sigma_i) = (\tau_i)(\mu,\mu,\ldots,\mu)$ and so $(\sigma_i)\sim (\tau_i)$ as desired.

For the other direction, assume that $(\sigma_i)\sim (\tau_i)$. Then there is a permutation $\mu\in C_{S_n}(\phi_{(\sigma_i)}(F_m^0)$ and a permutation $\sigma\in S_n$ such that $ \sigma_i\mu = \sigma\tau_i\sigma^{-1}$. We can assume without loss of generality that $\sigma=Id$, as this will not change the isomorphism type of $Q_{(\tau_i)}$ as an $F_m^0$-set. The group $F_m^0$ is generated by elements of the form $z_i^{-j}z_1^j$. We have that $\tau_i^{-j}\tau_1^j = (\sigma_i\mu)^{-j}(\sigma_1\mu)^j$. We use the fact that $\mu$ commutes with $\phi_{(\sigma_i)}(F_m^0)$ and we get
$$(\sigma_1\mu)^j = \sigma_1\mu\sigma_1^{-1}\sigma_1^2\mu\sigma_1^{-2}\cdots \sigma_1^j\mu\sigma_1^{-j}\sigma_1^j = $$
$$\sigma_i\mu\sigma_i^{-1}\sigma_i^2\mu\sigma_i^{-2}\cdots\sigma_i^j\mu\sigma_i^{-j}\sigma_1^j = (\sigma_i\mu)^j\sigma_i^{-j}\sigma_1^j,$$ where we used the fact that $\sigma_i^a$ and $\sigma_1^a$ conjugate $\mu$ in the same way, since $\mu$ commutes with $\sigma_1^a\sigma_i^{-a}$.
We thus get 
$$\tau_i^{-j}\tau_1^j = (\sigma_i\mu)^{-j}(\sigma_1\mu)^j = (\sigma_i\mu)^{-j}(\sigma_i\mu)^j\sigma_i^{-j}\sigma_1^j=\sigma_i^{-j}\sigma_1^j,$$
and the restrictions of $\phi_{(\sigma_i)}$ and $\phi_{(\tau_i)}$ to $F_m^0$ are equal as desired. 
\end{proof}
The above lemma together with Proposition \ref{prop:splitting1} show that group actions of $F_m^0$ on finite sets will play important role in studying the algebras $\Ainf_{\Ga_0}$. 
\begin{definition}
Let $S$ be a finite $F_m^0$-set, and let $t\in \Z$. We denote by $z_1^t(S)$ the set $S$ with the $F_m^0$-action $g\cdot s = z_1^{-t}gz_1^ts$.
\end{definition}
\begin{remark}
It is easy to see that $S$ is a transitive $F_m^0$-set if and only if $z_1^t(S)$ is transitive. 
\end{remark}
\begin{definition} We will call a transitive $F_m^0$-set $S$ strongly finite if there exists an integer $k>0$ such that $z_1^k(S)\cong S$. We write $SFT(\Ga_0)$ for the set of isomorphism classes of all strongly finite transitive $F_m^0$-sets. We will write $SFT=SFT(\Ga_0)$ when $\Ga_0$ is clear from the context. 
\end{definition}
Let now $S = \sqcup_{\Ow\in SFT} \Ow^{a(\Ow)}$ be a finite $F_m^0$-set. Here $a(\Ow)$ are integers that count how may orbits of type $\Ow$ appear in $S$. Since $S$ is finite, almost all of the $a(\Ow)$ are zero. We claim the following:
\begin{lemma}\label{lem:extension}
The $F_m^0$-action on $S$ can be extended to $F_m$ if and only if $a(\Ow) = a((z_1)(\Ow))$ for every $\Ow\in SFT$. 
\end{lemma}
\begin{proof}
If $S$ has a structure of an $F_m$-set then the action of $z_1$ provides a bijection between $\Ow^{a(\Ow)}$ and $(z_1(\Ow))^{a((z_1)(\Ow))}$, and in particular the exponents must be equal. 
In the other direction, assume that the condition of the lemma is satisfied. We can then write $S$ as a disjoint union of subsets of the form 
$S'=\Ow\sqcup z_1(\Ow)\sqcup\cdots\sqcup z_1^{k-1}(\Ow)$, where $z_1^k(\Ow)\cong \Ow$ and $\Ow\in SFT$. It will thus be enough to define an $F_m$-action on this subset. Since $F_m = F_m^0\cdot \langle z_1\rangle$, we just need to write down the $z_1$-action. 
Write $\Phi:z_1^k(\Ow)\to \Ow$ for an isomorphism of $F_m^0$-sets. We then define an action of $z_1$ on $S'$ in the following way:
$$z_1(f) = \begin{cases} f\in z_1^{i+1}(\Ow) & \text{ if } f\in z_1^i(\Ow), i<k-1 \\ \Phi(f)\in \Ow &  \text{ if } f\in z_1^{k-1}(\Ow)\end{cases},$$ where we used the fact that the underlying set of $z^t(\Ow)$ is the same as $\Ow$ for every $t\in\Z$. 
A direct verification shows that this gives a well defined action of $F_m$. 
\end{proof}
The cyclic group $\langle z_1\rangle$ thus acts on $SFT$ with finite stabilizers. We write $OSFT(\Ga_0)$ for a set of representatives of the $\langle z_1\rangle$-orbits in $SFT$. We write $OSFT=OSFT(\Ga_0)$ if $\Ga_0$ is clear from the context. 
For every $\Ow\in OSFT$ we write $k(\Ow)$ for the minimal positive integer such that $z_1^k(\Ow)\cong \Ow$, and we fix an isomorphism $\Phi(\Ow):x^{k(\Ow)}(\Ow)\to \Ow$.
The choice of $\Phi(\Ow)$ induces an automorphism of $\Aut_{F_m^0}(\Ow)$. Indeed, the fact that $z_1^k(\Ow)$ and $\Ow$ have the same underlying set induces an isomorphism $\Aut_{F_m^0}(\Ow)\cong \Aut_{F_m^0}(z_1^k(\Ow))$, given by sending an automorphism to itself. Conjugation with $\Phi(\Ow)$ then induces an isomorphism $\Aut_{F_m^0}(z_1^k(\Ow))\cong \Aut_{F_m^0}(\Ow)$. We denote the composition of the two isomorphisms by 
\begin{equation}\label{eq:defrho}\rho(\Ow):\Aut_{F_m^0}(\Ow)\to \Aut_{F_m^0}(\Ow).\end{equation} A different choice of $\Phi(\Ow)$ would just change $\rho(\Ow)$ by an inner automorphism. In other words, the class of $\rho(\Ow)$ in $\text{Out}(\Aut_{F_m^0}(\Ow))$ does not depend on the particular choice of $\Phi$. 

Lemma \ref{lem:extension} has the following corollary.
\begin{proposition}\label{prop:finitefm0sets}
Up to isomorphism, all finite $F_m^0$-sets that admit an extension to an $F_m$-set are of the form $$Y_a=\bigsqcup_{\Ow\in OSFT}(\Ow\sqcup \cdots\sqcup z_1^{k(\Ow)-1}(\Ow))^{a(\Ow)},$$ and are therefore in one to one correspondence with functions $a:OSFT\to \N$ that admit only finitely many non-zero values. 
\end{proposition}
\begin{proof}
The only non-trivial part is the fact that all the isomorphism classes of $F_m^0$-orbits that appear in finite $F_m$-sets are in $SFT$. But this follows from the fact that the action of $z_1$ permutes these orbits, and there are only finitely many of them. 
\end{proof}
The proof of Lemma \ref{lem:extension} gives a canonical way to extend the $F_m^0$-action on $Y_a$ to an $F_m$-action. We fix this extension henceforth, and we write $G_a = \Aut_{F_m^0}(Y_a)$. The figure below shows a schematic description of the $F_m$-set $(\Ow\sqcup \cdots\sqcup z_1^{k(\Ow)-1}(\Ow))^{a(\Ow)}$:
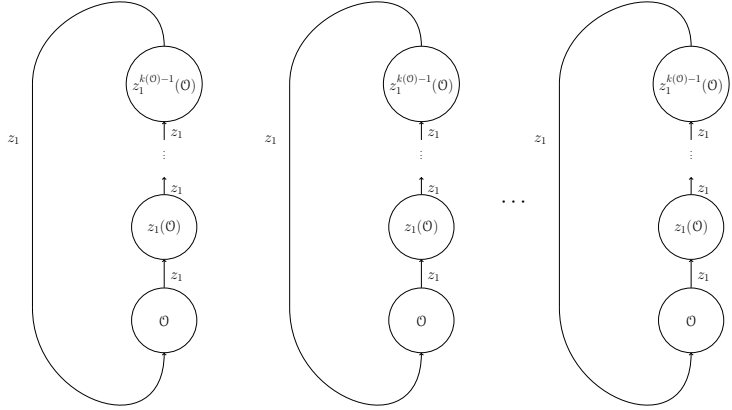
\begin{center}\scalebox{0.33}{
\begin{tikzpicture}
	\begin{pgfonlayer}{nodelayer}
		\node [style=Circle1] (0) at (-3.75, -3.75) {\Large{$z_1(\mathcal{O})$}};
		\node [style=Circle1] (1) at (-3.75, -7.5) {\Large{$\mathcal{O}$}};
		\node [style=Circle1] (3) at (-3.75, 2) {\Large{$z_1^{k(\mathcal{O})-1}(\mathcal{O})$}};
		\node [style=none] (4) at (-3.75, -0.75) {$\vdots$};
		\node [style=none] (6) at (-3.75, -1.75) {};
		\node [style=none] (7) at (-9, 2) {};
		\node [style=none] (8) at (-9, -7) {};
		\node [style=none] (9) at (-3.75, -0.25) {};
		\node [style=none] (12) at (-9.75, -0.25) {\Large{$z_1$}};
		\node [style=none] (37) at (10.25, -2.75) {\Huge{$\cdots$}};
		\node [style=none] (62) at (-3.25, -5.75) {\Large{$z_1$}};
		\node [style=none] (63) at (-3.25, -2.25) {\Large{$z_1$}};
		\node [style=none] (64) at (-3.25, 0) {\Large{$z_1$}};
		\node [style=Circle1] (65) at (6.5, -3.75) {\Large{$z_1(\mathcal{O})$}};
		\node [style=Circle1] (66) at (6.5, -7.5) {\Large{$\mathcal{O}$}};
		\node [style=Circle1] (67) at (6.5, 2) {\Large{$z_1^{k(\mathcal{O})-1}(\mathcal{O})$}};
		\node [style=none] (68) at (6.5, -0.75) {$\vdots$};
		\node [style=none] (69) at (6.5, -1.75) {};
		\node [style=none] (70) at (1.25, 2) {};
		\node [style=none] (71) at (1.25, -7) {};
		\node [style=none] (72) at (6.5, -0.25) {};
		\node [style=none] (73) at (0.5, -0.25) {\Large{$z_1$}};
		\node [style=none] (74) at (7, -5.75) {\Large{$z_1$}};
		\node [style=none] (75) at (7, -2.25) {\Large{$z_1$}};
		\node [style=none] (76) at (7, 0) {\Large{$z_1$}};
		\node [style=Circle1] (77) at (17.25, -3.75) {\Large{$z_1(\mathcal{O})$}};
		\node [style=Circle1] (78) at (17.25, -7.5) {\Large{$\mathcal{O}$}};
		\node [style=Circle1] (79) at (17.25, 2) {\Large{$z_1^{k(\mathcal{O})-1}(\mathcal{O})$}};
		\node [style=none] (80) at (17.25, -0.75) {$\vdots$};
		\node [style=none] (81) at (17.25, -1.75) {};
		\node [style=none] (82) at (12, 2) {};
		\node [style=none] (83) at (12, -7) {};
		\node [style=none] (84) at (17.25, -0.25) {};
		\node [style=none] (85) at (11.25, -0.25) {\Large{$z_1$}};
		\node [style=none] (86) at (17.75, -5.75) {\Large{$z_1$}};
		\node [style=none] (87) at (17.75, -2.25) {\Large{$z_1$}};
		\node [style=none] (88) at (17.75, 0) {\Large{$z_1$}};
	\end{pgfonlayer}
	\begin{pgfonlayer}{edgelayer}
		\draw [style=new edge style 0] (1) to (0);
		\draw [style=new edge style 0] (0) to (6.center);
		\draw [style=new edge style 0, in=270, out=-90, looseness=1.75] (8.center) to (1);
		\draw [style=new edge style 0] (9.center) to (3);
		\draw [in=90, out=90, looseness=1.50] (3) to (7.center);
		\draw (7.center) to (8.center);
		\draw [style=new edge style 0] (66) to (65);
		\draw [style=new edge style 0] (65) to (69.center);
		\draw [style=new edge style 0, in=270, out=-90, looseness=1.75] (71.center) to (66);
		\draw [style=new edge style 0] (72.center) to (67);
		\draw [in=90, out=90, looseness=1.50] (67) to (70.center);
		\draw (70.center) to (71.center);
		\draw [style=new edge style 0] (78) to (77);
		\draw [style=new edge style 0] (77) to (81.center);
		\draw [style=new edge style 0, in=270, out=-90, looseness=1.75] (83.center) to (78);
		\draw [style=new edge style 0] (84.center) to (79);
		\draw [in=90, out=90, looseness=1.50] (79) to (82.center);
		\draw (82.center) to (83.center);
	\end{pgfonlayer}
\end{tikzpicture}}
\end{center}

\begin{lemma} The action of $z_1$ on $Y_a$ induces, by conjugation, an automorphism of $G_a$.
\end{lemma}
\begin{proof}
This follows from the fact that if $\beta:Y_a\to Y_a$ is an $F_m^0$-map, $y\in Y_a$ and $g\in F_m^0$ then we have 
$$z_1\beta z_1^{-1}(gy) = z_1\beta(z_1^{-1}gz_1z_1^{-1}y) = z_1z_1^{-1}gz_1\beta(z_1^{-1}y) = gz_1\beta(z_1^{-1}y)=g(z_1\beta  z_1^{-1})(y),$$ and $z_1\beta z_1^{-1}\in G_a$ as desired. 
\end{proof}
\begin{definition}
We will denote the above automorphism of $G_a$ by $\nu_a$. 
\end{definition}
Write now $$FOSFT(\Ga_0) = \{a:OSFT(\Ga_0)\to \N| \; a\text{ admits only finitely many non-zero values}\}.$$
We will write $FOSFT(\Ga_0)=FOSFT$ when $\Ga_0$ is clear from the context. 
Let $a,b\in FOSFT$. The pointwise addition $a+b$ is in $FOSFT$ as well. We claim the following:
\begin{lemma}\label{lem:decYa}
We have an isomorphism $Y_a\sqcup Y_b\cong Y_{a+b}$ of $F_m$-sets. This isomorphism gives a natural inclusion $G_a\times G_b\subseteq G_{a+b}$. 
The automorphism $\nu_{a+b}$ stabilises this subgroup, and $\nu_{a+b}|_{G_a\times G_b} = \nu_a\times \nu_b$.
\end{lemma}
\begin{proof} We have a natural identification of $F_m^0$-sets $$Y_a \sqcup Y_b=(\bigsqcup_{\Ow\in OSFT}(\Ow\sqcup \cdots\sqcup z_1^{k(\Ow)-1}(\Ow))^{a(\Ow)})\bigsqcup (\bigsqcup_{\Ow\in OSFT}(\Ow\sqcup \cdots\sqcup z_1^{k(\Ow)-1}(\Ow))^{b(\Ow)})\cong $$ $$\bigsqcup_{\Ow\in OSFT}(\Ow\sqcup \cdots\sqcup z_1^{k(\Ow)-1}(\Ow))^{(a+b)(\Ow)}=Y_{a+b}$$
By the particular way we extended the action to $F_m$ in the proof of Lemma \ref{lem:extension} we see that this isomorphism also commutes with the action of $z_1$, and therefore this is an $F_m$-isomorphism as well. Since $\nu_a,\nu_b$, and $\nu_{a+b}$, were defined using the action of $z_1$, we get the last statement.
\end{proof}
We conclude this section with the following proposition:
\begin{proposition}\label{prop:decga0} We have an isomorphism $$\Ainf_{\Ga_0}\cong \bigoplus_{a\in FOSFT} (KG_a)_{G_a},$$ where the action of $G_a$ on $KG_a$ is given by $g_1\cdot g_2 = g_1g_2\nu_a(g_1)^{-1}$. The multiplication is induced by the natural inclusion $G_a\times G_b\subseteq G_{a+b}$. The squared norm of $\ol{g}\in G_a$ is equal to the order of the fixed point subgroup $|G_a^{\nu_ac_g}|$, where $c_g$ denotes conjugation by $g$.  
\end{proposition}
\begin{proof}
We have seen in Proposition \ref{prop:splitting1} that $\Ainf_{\Ga_0}\cong \bigoplus_{n\geq 0}(KS_n^m)_{S_n}$. 
In Lemma \ref{lem:similarity} we have seen that $(\sigma_i)\sim (\tau_i)$ if and only if they define isomorphic $F_m^0$-actions. By Proposition \ref{prop:finitefm0sets} all the $F_m^0$-sets that are extendable to $F_m$-sets are of the form $Y_a$, for some $a\in FOSFT$.  By Lemma \ref{lem:similarity} and Equation \ref{eq:splittingKG} applied to the case where $G=S_n^m$ and $H=S_n$ we have 
$$(KS_n^m)_{S_n}\cong \bigoplus_Y (K\Aut_{F_m^0}(Y))_{\Aut_{F_m^0}(Y)},$$ where $Y$ runs over all the isomorphism classes of $F_m^0$-sets of cardinality $n$ that are extendable to $F_m$, and the twisted action of $\Aut_{F_m^0}(Y)$ on the group algebra is defined using the action of $z_1$. Since every finite $F_m^0$-set that is extendable to an $F_m$-set is isomorphic to $Y_a$ for a unique $a$, we get the result. The statement about multiplication follows from the fact that multiplication is just given by taking disjoint unions of $F_m$-sets. The last statement about the squared norm follows from Equation \ref{eq:sqnorm}. 
\end{proof}
The results of the first part of this section give us an orthonormal basis for $(KS_n^m)_{S_n}$. The $\Z$-span of this basis gives us a lattice. In order to prove that we get a PSH-algebra indeed we need to show that all our structure constants are non-negative integers. 
This will be done in the next section, by carefully analysing the groups $G_a$ and their automorphisms $\nu_a$. 

\section{The groups $G_a$}\label{sec:grpGa}
Fix now $a\in FOSFT$. Recall that $G_a$ is the automorphism group of $$Y_a = \bigsqcup_{\Ow\in OSFT}(\Ow\sqcup \cdots\sqcup z_1^{k(\Ow)-1}(\Ow))^{a(\Ow)}$$ as an $F_m^0$-set. We claim the following:
\begin{lemma}
We have \begin{equation}\label{eq:dec1}G_a \cong \prod_{\Ow\in OSFT}(S_{a(\Ow)}\ltimes \Aut_{F_m^0}(\Ow)^{a(\Ow)})^{k(\Ow)},\end{equation} where the action of $S_{a(\Ow)}$ on $\Aut_{F_m^0}(\Ow)^{a(\Ow)}$ is given by permuting the factors. 
\end{lemma}
\begin{proof}
For $\Ow\in OSFT$ and $0\leq i\leq k(\Ow)-1$ write $Z(i,\Ow) = z_1^i(\Ow)^{a(\Ow)}$.
It holds that \begin{equation} Y_a = \bigsqcup_{\Ow\in OSFT}\bigsqcup_{i=0}^{k(\Ow)-1} Z(i,\Ow).\end{equation}
Moreover, every $g\in G_a$ fixes the above decomposition, and we therefore have 
\begin{equation}\label{eq:decomposeG1}G_a\cong \prod_{\Ow\in OSFT}\prod_{i=0}^{k(\Ow)-1} \Aut_{F_m^0}(Z(i,\Ow)).\end{equation}

Fix now $i$ and $\Ow$ and consider the group $\Aut_{F_m^0}(Z(i,\Ow))$.
Every permutation in $S_{a(\Ow)}$ gives an $F_m^0$-automorphism $Z(i,\Ow)$ by permuting the orbits. 
Every automorphism in $\Aut_{F_m^0}(Z(i,\Ow))$ permutes the orbits, and we thus see that we get a split surjection $\Aut_{F_m^0}(Z(i,\Ow))\to S_{a(\Ow)}$. 
The kernel of this homomorphism is the group of all automorphisms of $Z(i,\Ow)$ that preserves every orbit, and is thus isomorphic to $\Aut_{F_m^0}(z_1^i(\Ow))^{a(\Ow)}\cong \Aut_{F_m^0}(\Ow)^{a(\Ow)}$. 

Thus, $\Aut_{F_m^0}(z_1^i(\Ow)^{a(\Ow)}) \cong S_{a(\Ow)}\ltimes \Aut_{F_m^0}(\Ow)^{a(\Ow)}$. Using the decomposition in Equation \ref{eq:decomposeG1}, we get the claim of the lemma.
\end{proof}
Next, we write down explicitly the automorphism $\nu_a$ of $G_a$. 
Since the action of $z_1$ fixes the subsets $Z=(\Ow\sqcup z_1(\Ow)\sqcup\cdots\sqcup z_1^{k(\Ow)-1}(\Ow))^{a(\Ow)}$ it is enough to consider the action of $\nu_a$ on the automorphism group of these sets. 
The set $Z$ is the disjoint union of $k(\Ow)\cdot a(\Ow)$ sets that are all naturally identified with $\Ow$. We will use subscript to distinguish elements in the different orbits. Thus, for $f$ in the $l$-th copy of $z_1^i(\Ow)$ we will write $f_{il}$. With this notation, it holds that $z_1(f_{il}) = f_{i+1,l}$ for $i<k(\Ow)-1$ and $z_1(f_{k(\Ow)-1,l}) = \Phi(\Ow)(f)_{0,l}$. 

Let now $\beta:Z\to Z$ be given by $(\sigma_i,\lambda_{i1},\ldots, \lambda_{ia(\Ow)})_{i=0}^{k(\Ow)-1}$ where $\sigma_i\in S_{a(\Ow)}$ and $\lambda_{ij}\in \Aut_{F_m^0}(\Ow)$. We calculate $\nu_a(\beta)$ = $z_1\beta z_1^{-1}$: 
$$z_1\beta\xi^{-1}(f_{il}) = z_1\beta(f_{i-1,l}) = z_1\la_{i-1,l}(f)_{i-1,\sigma_{i-1}(l)} = \la_{i-1,l}(f)_{i,\sigma_{i-1,}(l)} \text{ where } i>0$$
$$z_1\beta z_1^{-1}(f_{0l}) = z_1\beta (\Phi(\Ow)^{-1}(f))_{k(\Ow)-1,l} = z_1 \la_{k(\Ow)-1,l}(\Phi(\Ow)^{-1}(f))_{k(\Ow)-1,\sigma_{k(\Ow)-1}(l)} =$$ $$(\Phi(\Ow)\la_{k(\Ow)-1,l}\Phi(\Ow)^{-1}(f))_{0,\sigma_{k(\Ow)-1}(l)}$$
The map $\la\mapsto \Phi(\Ow)\la\Phi(\Ow)^{-1}$ is exactly $\rho(\Ow)$ defined in Equation \ref{eq:defrho}. Write $g_i = (\sigma_i,\lambda_{i1},\ldots, \lambda_{ia(\Ow)})\in S_{a(\Ow)}\ltimes \Aut_{F_m^0}(\Ow)^{a(\Ow)}$. Then $\beta = (g_0,\ldots, g_{k(\Ow)-1})$. The above calculation gives the following lemma:
\begin{lemma}
The restriction of the automorphism $\nu_a$ to $(S_{a(\Ow)}\ltimes \Aut_{F_m^0}(\Ow)^{a(\Ow)})^{k(\Ow)}$ is given explicitly by the formula
$$\nu_a(g_0,\ldots, g_{k(\Ow)-1}) =(\xi_{a(\Ow)}(g_{k(\Ow)-1}),g_0,\ldots, g_{k(\Ow)-2} ),$$
where $\xi_{a(\Ow)}:S_{a(\Ow)}\ltimes \Aut_{F_m^0}(\Ow)^{a(\Ow)}\to S_{a(\Ow)}\ltimes \Aut_{F_m^0}(\Ow)^{a(\Ow)}$ is the automorphism given by $$(\sigma,\la_1,\ldots, \la_{a(\Ow)})\mapsto (\sigma, \rho(\Ow)(\la_1),\ldots, \rho(\Ow)(\la_{a(\Ow)})).$$
\end{lemma}
The action of $\nu_a$ on $G_a$ respects the direct product decomposition in Equation \ref{eq:dec1}. We thus get 
$$(KG_a)_{G_a}\cong \bigotimes_{\Ow\in OSFT} ((KS_{a(\Ow)}\ltimes \Aut_{F_m^0}(\Ow)^{a(\Ow)})^{k(\Ow)})_{(S_{a(\Ow)}\ltimes \Aut_{F_m^0}(\Ow)^{a(\Ow)})^{k(\Ow)}}.$$
Proposition \ref{prop:decga0} now gives 
$$\Ainf_{\Ga_0}\cong \bigoplus_{a\in FOSFT} (KG_a)_{G_a} \cong$$ $$ \bigoplus_{a\in FOSFT}  \bigotimes_{\Ow\in OSFT} ((KS_{a(\Ow)}\ltimes \Aut_{F_m^0}(\Ow)^{a(\Ow)})^{k(\Ow)})_{(S_{a(\Ow)}\ltimes \Aut_{F_m^0}(\Ow)^{a(\Ow)})^{k(\Ow)}}\cong $$
$$\bigotimes_{\Ow\in OSFT}\bigoplus_{n\geq 0} ((KS_n\ltimes \Aut_{F_m^0}(\Ow)^n)^{k(\Ow)}_{(S_n\ltimes \Aut_{F_m^0}(\Ow)^n)^{k(\Ow)}},$$ where we have used the fact that any $a\in FOSFT$ is determined by its values on the different elements of $OSFT$. 
For $\Ow\in OSFT$ write
$$\Ainf_{\Ga_0,\Ow} = \bigoplus_{n\geq 0} ((KS_n\ltimes \Aut_{F_m^0}(\Ow)^n)^{k(\Ow)}_{(S_n\ltimes \Aut_{F_m^0}(\Ow)^n)^{k(\Ow)}}.$$ 
We claim the following:
\begin{lemma}
For every $\Ow\in OSFT$  it holds that $\Ainf_{\Ga_0,\Ow}$ is a Hopf subalgebra of $\Ainf_{\Ga_0}$. 
\end{lemma}
\begin{proof}
Write $FOSFT_{\Ow}:= \{a: OSFT\to \N | a(\Ow')=0 \text{ for all } \Ow'\neq \Ow \}$.
We can write $\Ainf_{\Ga_0,\Ow} \cong \bigoplus_{a\in FOSFT_{\Ow}} (KG_a)_{G_a}$, 
By Proposition \ref{prop:decga0} the multiplication is induced by the natural inclusion $G_a\times G_b\to G_{a+b}$. For $a,b\in FOSFT$ it holds that $a+b\in FOSFT_{\Ow}$ if and only if $a,b\in FOSFT_{\Ow}$. This implies that $\Ainf_{\Ga_0,\Ow}$ is closed under multiplication. Since the comultiplication is dual to the multiplication, it holds that for $c\in FOSFT$ we have $$\Delta((KG_c)_{G_c}) \subseteq \bigoplus _{a+b=c} (KG_a)_{G_a}\ot (KG_b)_{G_b}.$$ But if $c\in FOSFT_{\Ow}$ and $a+b=c$ then $a,b\in FOSFT_{\Ow}$ as well. This shows that $\Ainf_{\Ga_0,\Ow}$ is closed under comultiplication as well, and is therefore a sub-bialgebra. Since it is connected, it is also a Hopf subalgebra. 
\end{proof}

We summarize this discussion with the following result:
\begin{proposition}
We have an isomorphism of Hopf algebras $$\Ainf_{\Ga_0}\cong \bigotimes_{\Ow\in OSFT} \Ainf_{\Ga_0,\Ow}.$$
\end{proposition}
We would like to get rid of the exponent $k(\Ow)$ in the description of $\Ainf_{\Ga_0,\Ow}$. For this, we prove a more general result. Let $G$ be a finite group with an automorphism $\xi:G\to G$. Let $k>0$ be an integer, and let $\nu:G^k\to G^k$ be the automorphism $(g_0,\ldots, g_{k-1})\mapsto (\xi(g_{k-1}),g_0,\ldots, g_{k-2})$. We claim the following:
\begin{lemma}
With the above notations the inverse order multiplication map $G^k\to G$ $(g_0,\ldots, g_{k-1})\mapsto g_{k-1}\cdots g_0$ induces an isomorphism $(KG^k)_{G^k}\to (KG)_G$, where $G^k$ acts on $KG^k$ by $a\cdot b = ab\nu(a)^{-1}$ and $G$ acts on $KG$ by $a\cdot b= ab\xi(a)^{-1}$.
\end{lemma}
\begin{proof}
Write $G^{k-1}\subseteq G^k$ for the subgroup $\{(1,g_1,\ldots, g_{k-1})\}$. It holds that $(KG^k)_{G^{k-1}}\cong KG\ot_G KG\ot_G\cdots \ot_G KG\cong KG$. It is then immediate to check that the first copy of $G$ in $G^k$ acts on the resulting $KG$ by the automorphism $\xi$. 
\end{proof}
Using the above lemma we can write 
$$\Ainf_{\Ga_0,\Ow}\cong \bigoplus_{n\geq 0} (KS_n\ltimes \Aut_{F_m^0}(\Ow)^n)_{S_n\ltimes \Aut_{F_m^0}(\Ow)^n},$$
where the action of $S_n\ltimes \Aut_{F_m^0}(\Ow)^n$ is the twisted $\xi_n$-action. 
By following the isomorphisms we see that multiplication in $\Ainf_{\Ga_0,\Ow}$ is induced by the inclusion of groups $$(S_n\ltimes \Aut_{F_m^0}(\Ow)^n)\times (S_m\ltimes \Aut_{F_m^0}(\Ow)^m)\to S_{n+m}\ltimes \Aut_{F_m^0}(\Ow)^{n+m},$$ and that the squared norm of $\ol{g}\in (KS_n\ltimes \Aut_{F_m^0}(\Ow)^n)_{S_n\ltimes \Aut_{F_m^0}(\Ow)^n}$ is equal to the order of the fixed point subgroup $|S_n\ltimes (\Aut_{F_m^0}(\Ow)^n)^{\xi_nc_g}|$, where $c_g$ denotes conjugation by $g$.

%Thus, in order to find a PSH-basis for $\Ainf$ it will be enough to find a PSH-basis for $\Ainf_{\Ga_0,\Ow}$. 

\section{The algebra $\Ainf_{\Ga_0,\Ow}$}\label{sec:mains}
As before, let $\Ga_0$ be an irreducible graph, and let $\Ow\in OSFT(\Ga_0)$. 
Write $B = \Aut_{F_m^0}(\Ow)$ and $\rho=\rho(\Ow):B\to B$. The automorphism $\rho$ induces the automorphism $\xi=\xi_n:S_n\ltimes B^n\to S_n\ltimes B^n$. 
Let $\Irr(B) = \{[W_1],\ldots, [W_b]\}$. In Section \ref{sec:prelim} we have seen that irreducible representations of the wreath product $S_n\ltimes B^n$ are in one to one correspondence with tuples $(a_i,\la_i)$ where $\sum_i a_i=n$ and $\la_i\parti a_i$, and we denoted the corresponding representation by $W_{(a_i,\la_i)}$. 
Assume without loss of generality that $[W_1],\ldots, [W_a]$ are representatives of the distinct $\rho$-orbits in $\Irr(B)$. Write $l_i$ for the length of the $\rho$-orbit of $[W_i]$, $(i=1,\ldots a)$. Thus $\Irr(B)=\{[(\rho^j)^*(W_i)]\}_{i=1,\ldots, a, j=0,\ldots, l_i-1}$. Fix isomorphisms $T_{W_i}:(\rho^{l_i})^*(W_i)\to W_i$. 
Write 
$$W_{i,\la,c} = \Ind_{S_{c}^{l_i}\ltimes B^{cl_i}}^{S_{cl_i}\ltimes B^{cl_i}}(\S_{\la}^{\boxtimes l_i}\ot W_i^{\ot c}\ot\cdots\ot (\rho^{l_i-1})^*(W_i)^{\ot c})\in \Rep(S_{cl_i}\ltimes B^{cl_i}),$$
where $i\in \{1,\ldots, a\}$, $c\in \N$ and $\la\parti c$. Then the discussion in Section \ref{sec:prelim} implies that  every irreducible $\xi$-invariant representation of $S_n\ltimes B^n$ is of the form $$\Ind_{\prod S_{c_i}^{l_i}\ltimes B^n}^{S_n\ltimes B^n}\bigotimes_i W_{i,\la_i,c_i}$$ for some integers $c_i$ and partitions $\la_i\parti c_i$ such that $\sum c_il_i = n$. 
Moreover, the automorphism $\xi$ stabilizes the subgroups $S_{cl_i}\ltimes B^{cl_i}$ and it holds that we have an isomorphism $T_{i,\la_i,c}:\xi^*(W_{i,\la_i,c})\to W_{i,\la_i,c}$ given by $$1\ot (s_0\ot\cdots\ot s_{l_i-1}\ot w_{01}\ot\cdots w_{0c}\ot w_{11}\ot\cdots\ot w_{1c}\ot\cdots w_{l_i-1,1}\ot\cdots w_{l_i-1,c})\mapsto $$
$$\tau\ot (s_{l_i-1}\ot s_0\ot\cdots\ot s_{l_i-2}\ot T_{W_i}(w_{l_i-1,1})\ot\cdots\ot T_{W_i}(w_{l_i-1,c})\ot w_{01}\ot\cdots\ot w_{0c}\ot\cdots w_{l_i-2,1}\ot\cdots w_{l_i-2,c}),$$ where $\tau\in S_{cl_i}$ is given by $\tau(j) = j-c\text{ mod } cl_i$, $s_j\in \S_{\la_i}$, and $w_{jk}\in (\rho^j)^*(W_i)$. In order to ease notations, we will write \begin{equation}\label{eq:Tilc} T_{i,\la_i,c}(1\ot (s_0\ot\cdots\ot s_{l_i-1}\ot w)  =\tau\ot (s_{l_i-1}\ot s_0\ot\cdots\ot s_{l_i-2}\ot \wt{T_{i,\la_i,c}}(w)).\end{equation}  
For every $i=1,\ldots a$ we write now $\Ainf_{\Ga_0,\Ow,i}$ for the subspace spanned by all elements of the form $T_{i,\la,c}$, where $c\in \N$ and $\la\parti c$. 
\begin{lemma}
The subspaces $\Ainf_{\Ga_0,\Ow,i}$ are sub Hopf-algebras of $\Ainf_{\Ga_0,\Ow}$, and we have $$\Ainf_{\Ga_0,\Ow}\cong \bigotimes_i \Ainf_{\Ga_0,\Ow,i},$$ where the tensor product is taken over all the $\rho$-orbits in $\Irr(\Aut_{F_m^0}(\Ow))$. Moreover, $\Ainf_{\Ga_0,\Ow,i}$ has an orthonormal basis given by the elements $$\ol{T_{i,\la,c}}=\frac{1}{\dim(W_{i,\la,c})}T_{i,\la,c}.$$
\end{lemma}
\begin{proof}
The fact that $\Ainf_{\Ga_0,\Ow,i}$ is a subalgebra of $\Ainf_{\Ga_0,\Ow}$ follows from the discussion in the end of Section \ref{sec:grpGa}. Since comultiplication is dual to multiplication, it is easy to see that this is also a Hopf-subalgebra. The fact that the elements $\ol{T_{i,\la,c}}$ form an orthonormal set follows from the the results of Section \ref{sec:decomposing}.
\end{proof}
\begin{definition}
Let $\H_{\Ga_0,\Ow,i}=\oplus_{\la,c}\Z \ol{T_{i,\la,c}}$ be the $\Z$-lattice in $\Ainf_{\Ga_0,\Ow,i}$ spanned by the above basis. 
Let $\H_{\Ga_0}=\bigotimes_{\Ow,i}\H_{\Ga_0,\Ow,i}$.
\end{definition}
The following is the main theorem of this paper, which partially answers Question 1 in \cite{meirUIR}:
\begin{theorem}\label{thm:main}
The $\Z$-lattice $\H_{\Ga_0}$ with the basis given by all products of the basis elements of the subalgebras $\H_{\Ga_0,\Ow,i}$ and with the restrictions of the structure of $\Ainf$ is a PSH-algebra. 
\end{theorem}
\begin{proof}
Since $\Ainf$ is a self-adjoint Hopf algebra, and since $\H_{\Ga_0}$ has an orthonormal basis, the only thing that we need to prove is that all the structure constants are non-negative integers. 
The proof of the theorem will thus be completed once we prove the next result.
\end{proof}
\begin{lemma}\label{lem:LR}
Let $\la\parti c, \mu\parti  d, \nu\parti c+d$. Then $\langle \ol{T_{i,\la,c}}\cdot \ol{T_{i,\mu,d}},\ol{T_{i,\nu,c+d}}\rangle= c_{\la,\mu}^{\nu}$, the Littlewood-Richardson coefficient. 
\end{lemma}
\begin{proof}
Write $\Res^{S_{c+d}}_{S_c\times S_d}(\S_{\nu}) = (U^{\nu}_{\la,\mu}\ot \S_{\la}\boxtimes \S_{\mu})\oplus\ol{\S_{\nu}}$, where $\dim(U_{\la,\mu}^{\nu}) = c_{\la,\mu}^{\nu}$. We thus have $$\Hom_{S_c\times S_d}(\ol{\S_{\nu}},\S_{\la}\boxtimes \S_{\mu})=0.$$ We can then similarly write $$\Res^{S_{c+d}^{l_i}}_{S_c^{l_i}\times S_d^{l_i}}(\S_{\nu}^{\boxtimes l_i}) = (U^{\nu}_{\la,\mu})^{\boxtimes l_i}\ot (\S_{\la}\boxtimes \S_{\mu})^{\boxtimes l_i}\oplus \ol{\S_{\nu}^{\boxtimes l_i}}.$$
Write $\W = W_i^{\ot (c+d)}\ot\cdots\ot (\rho^{l_i-1})^*(W_i)^{\ot (c+d)}$, 
$N=B^{(c+d)l_i}$, $Q= S_{(c+d)l_i}$, and $Q_1 = S_{cl_i}\times S_{dl_i}$. Here $Q_1$ is embedded in $Q$ by letting $S_{cl_i}$ act on the numbers \begin{equation}\label{eq:R1}R_1:=\{j(c+d)+ k| 0\leq j< l_i, 0<k\leq c\}\end{equation} and  letting $S_{dl_i}$ act on the numbers \begin{equation}\label{eq:R2}R_2:=\{j(c+d)+c +k | 0\leq j\leq l_i, 0<k\leq d\}\end{equation} (we can think of this as dividing $\{1,\ldots, (c+d)l_i\}$ into segments of lengths $c,d,c,d\ldots, c,d$). 
The conditions of Lemma \ref{lem:technical} are fulfilled, where $Q_2 = S_{c+d}^{l_i}$, $Q_3 = Q_1\cap Q_2 = (S_c\times S_d)^{l_i}$, $V= \S_{\nu}^{\boxtimes l_i}$, $U= (\S_{\la}\boxtimes \S_{\mu})^{\boxtimes l_i}$, $\wt{V} = W_{i,\nu,c+d}$ and $\wt{U} = W_{i,\la,c}\boxtimes W_{i,\mu,d}$. Here $Q_2$ is embedded in $Q$ by letting the $j$th copy of $S_{c+d}$ act on the numbers $\{(j-1)(c+d) + k | 0<k\leq c+d\}$. The $W_{i,\la,c}\boxtimes W_{i,\mu,d}$-isotypic component in $\Res^{Q\ltimes N}_{Q_1\ltimes N}W_{i,\nu,c+d}$ is then equal to 
$$\Ind_{Q_3\ltimes N}^{Q_1\ltimes N}(U^{\nu}_{\la,\mu})^{\boxtimes l_i}\ot (\S_{\la}\ot \S_{\mu})^{\boxtimes l_i}\ot \W)\cong (U^{\nu}_{\la,\mu})^{\boxtimes l_i}\ot \Ind_{Q_3\ltimes N}^{Q_1\ltimes N}(\S_{\la}\ot \S_{\mu})^{\boxtimes l_i}\ot \W).$$
The linear automorphism $T_{i,\nu,c+d}:\xi^*(W_{i,\nu,c+d})\to W_{i,\nu,c+d}$ stabilizes this subspace, and is given explicitly by 
$$u_1\ot\cdots\ot u_{l_i}\ot 1\ot s_{11}\ot s_{12}\ot\cdots\ot s_{l_i1}\ot s_{l_i2}\ot w\mapsto $$ $$u_{l_i}\ot\cdots\ot u_{l_i-1}\ot \tau_{c+d}\ot s_{l_i1}\ot s_{l_i2}\ot\cdots\ot s_{l_i-1,1}\ot s_{l_i-1,2}\ot \wt{T_{i,\nu,c+d}}(w),$$ where $\wt{T_{i,\nu,c+d}}$ is given in Formula \ref{eq:Tilc} above, and the permutation $\tau_{c+d}\in S_{(c+d)l_i}$ is given by $\tau(j) = j-(c+d)$ mod $(c+d)l_i$.

The map $T_{i,\la,c}\ot T_{i,\mu,d}$ can also be written explicitly and is given by
$$u_1\ot\cdots\ot u_{l_i}\ot 1\ot s_{11}\ot s_{12}\ot\cdots\ot s_{l_i1}\ot s_{l_i2}\ot w\mapsto $$ $$u_{1}\ot\cdots\ot u_{l_i}\ot (\tau_c,\tau_d)\ot s_{l_i1}\ot s_{l_i2}\ot\cdots\ot s_{l_i-1,1}\ot s_{l_i-1,2}\ot (\wt{T_{i,\la,c}}\ot\wt{T_{i,\mu,d}})(w).$$ Here the permutation $(\tau_c,\tau_d)\in S_{cl_i}\times S_{dl_i}$ is given explicitly by$\tau_c(j) = j-c$ mod $cl_i$ and $\tau_d(j) = j-d$ mod $dl_i$. Since we embed $S_{cl_i}\times S_{dl_i}$ by letting $S_{cl_i}$ act on the set $R_1$ and by letting $S_{dl_i}$ act on the set $R_2$ defined in Equations \ref{eq:R1} and \ref{eq:R2} respectively, the permutation $(\tau_c,\tau_d)$ is equal to the permutation $\tau_{c+d}$ from the previous equation. We will denote the common value of these permutations by $\tau$ henceforth. 

We have an equality $\wt{T_{i,\la,c}}\ot\wt{T_{i,\mu,d}}= \wt{T_{i,\nu,c+d}}$. The composition $(T_{i,\la,c}\ot T_{i,\mu,d})T_{i,\nu,c+d}^*$ is therefore 
$$u_1\ot\cdots\ot u_{l_i}\ot 1\ot s_{11}\ot s_{12}\ot\cdots\ot s_{l_i1}\ot s_{l_i2}\ot w\mapsto $$ $$u_2\ot\cdots\ot u_{l_i}\ot u_1\ot 1\ot s_{11}\ot s_{12}\ot\cdots\ot s_{l_i1}\ot s_{l_i2}\ot w.$$ In other words, we can write this map as $L\ot \Id$ where $L:(U_{\la,\mu}^{\nu})^{\ot l_i}\to (U_{\la,\mu}^{\nu})^{\ot l_i}$ is given by cyclically permuting the tensors, and $\Id$ is the identity on $W_{i,\la,c}\ot W_{i,\mu,d}$. By Lemma \ref{lem:cyclic} the trace of this map is $\dim(W_{i,\la,c})\dim(W_{i,\mu,d})\cdot \dim(U_{\la,\mu}^{\nu})=  \dim(W_{i,\la,c})\dim(W_{i,\mu,d})\cdot c^{\nu}_{\la,\mu}$. Since $\ol{T_{i,\la,c}}= \frac{1}{\dim(W_{i,\la,c})}T_{i,\la,c}$, and similarly for $\mu$ and $\nu$ Lemma \ref{lem:innerprod} gives $\langle \ol{T_{i,\la,c}}\cdot \ol{T_{i,\mu,d}},\ol{T_{i,\nu,c+d}}\rangle= c_{\la,\mu}^{\nu}$ as desired. 
\end{proof}
Lemma \ref{lem:LR} shows that in fact we also have a parametrisation of the cuspidal elements in $\H_{\Ga_0}$. Indeed, $\H_{\Ga_0}$ splits as the tensor product of the algebras $\H_{\Ga_0,\Ow,i}$, and we have just seen that each one of these algebras is a basic PSH-algebra. 
\begin{theorem}\label{thm:cuspidals}
The cuspidal elements in the lattice $\H_{\Ga_0}$ are parametrised by pairs $(\Ow,\ol{[W]})$, where $\Ow$ is a strongly finite transitive $F_m^0$-set, and $\ol{[W]}\in \Irr(\Aut(F_m^0))/\langle \rho\rangle$. The cuspidal element that corresponds to $(\Ow,\ol{[W]})$ has degree $|\Ow|k(\Ow)l([W])$, where $l([W])$ is the cardinality of the $\rho$-orbit of $[W]$. \end{theorem}

\section{Examples}\label{sec:examples}
\subsection{Graphs with arbitrary large fundamental group ranks}
Let $m\geq 1$. We give here an example how graphs with fundamental group $F_m$ can occur, even when we start with relatively simple structure tensors. 
Assume that our type $((p_i,q_i))$ contains structure tensors $x_1$ and $x_2$ with $(p_1,q_1)=(1,2)$ and $(p_2,q_2) = (2,1)$. Such structure tensors occur, for example, when one considers Hopf algebras or Frobenius algebras. The invariant $$\Tr(x_1(x_1\ot 1)(x_1\ot 1\ot 1)\cdots (x_1\ot 1^{m-2})(x_2\ot 1^{m-2})\cdots x_2)$$ has an associated irreducible graph $\Ga_0$ with $\pi_1(\Ga_0)\cong F_m$. An alternative graph with fundamental group $F_m$ is given by simply taking a single tensor $x_3$ of type $(m,m)$ and taking its trace. 

\subsection{The case $m=1$}
Consider the algebraic structure that contains $k$ endomorphisms $T_i:W\to W$, $i=0,\ldots, k-1$. Such an algebraic structure can also be thought of as representations of the free algebra on $k$ generators. In this case, the connected diagrams just correspond to invariants of the form $\Tr(T_{i_1}T_{i_2}\cdots T_{i_n})$. Denote the adequate graph that corresponds to such an invariant by $\Ga_{i_1,\ldots, i_n}$. We have $\Ga_{i_1,\ldots, i_n} = \Ga_{i_n,i_1,\ldots i_{n-1}}$. Moreover, we have a covering $\Ga_{i_1,\ldots, i_n}\to \Ga_{j_1,\ldots, j_l}$ if and only if there is an $r>0$ such that $(i_1,\ldots i_n) = (j_1,\ldots, j_l)^r$ up to a cyclic permutation. The figure below shows an example of a covering of such graphs, that correspond to the invariants $\Tr(T_1T_2T_1T_2)$ and $\Tr(T_1T_2)$:
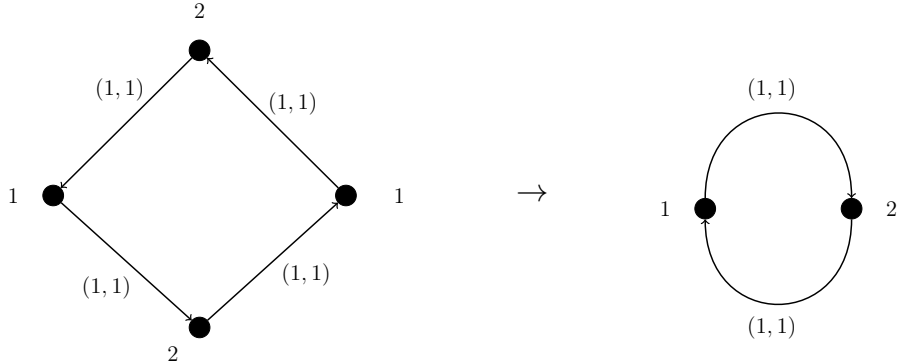
\begin{center}\scalebox{0.7}{
\begin{tikzpicture}
	\begin{pgfonlayer}{nodelayer}
		\node [style=graph node] (0) at (-6.25, 1) {};
		\node [style=graph node] (1) at (-3.5, 3.75) {};
		\node [style=graph node] (2) at (-0.75, 1) {};
		\node [style=graph node] (3) at (-3.5, -1.5) {};
		\node [style=none] (4) at (-7, 1) {1};
		\node [style=none] (5) at (-4, -2) {2};
		\node [style=none] (6) at (-3.5, 4.5) {2};
		\node [style=none] (7) at (0.25, 1) {1};
		\node [style=none] (8) at (2.75, 1) {\Large{$\to$}};
		\node [style=graph node] (9) at (6, 0.75) {};
		\node [style=graph node] (10) at (8.75, 0.75) {};
		\node [style=none] (11) at (5.25, 0.75) {1};
		\node [style=none] (12) at (9.5, 0.75) {2};
		\node [style=none] (13) at (-5, 3) {$(1,1)$};
		\node [style=none] (14) at (-5.25, -0.75) {$(1,1)$};
		\node [style=none] (15) at (-1.5, -0.5) {$(1,1)$};
		\node [style=none] (16) at (-1.75, 2.75) {$(1,1)$};
		\node [style=none] (17) at (7.25, 3) {$(1,1)$};
		\node [style=none] (18) at (7.25, -1.5) {$(1,1)$};
	\end{pgfonlayer}
	\begin{pgfonlayer}{edgelayer}
		\draw [style=new edge style 0] (0) to (3);
		\draw [style=new edge style 0] (3) to (2);
		\draw [style=new edge style 0] (2) to (1);
		\draw [style=new edge style 0] (1) to (0);
		\draw [style=new edge style 0, bend left=90, looseness=2.00] (10) to (9);
		\draw [style=new edge style 0, bend left=90, looseness=2.00] (9) to (10);
	\end{pgfonlayer}
\end{tikzpicture}}
\end{center}

Write $c_n:\{0,1,\ldots, k-1\}^n\to \{0,1,\ldots, k-1\}^n$ for the cyclic permutation. The vector space $\Ainf_n$ has a basis that is given by the orbit space $\{0,\ldots, k-1\}^n/\langle c_n\rangle$. A $c_n$-orbit in $\{0,1,\ldots, k-1\}^n$ gives a cuspidal element in $\Ainf_n$ if and only if it contains exactly $n$ elements. 

We thus see that in this case all the connected graphs are cycles. This implies that for every irreducible graph the parameter $m$ is equal to 1. Since in this case $F_m^0$ is the trivial group, the set of cuspidal elements in $\Ainf$ is in one to one correspondence with the irreducible graphs. 

A particular instance of this is the case where $k=q=p^l$, a prime power. In this case the algebra $\Ainf$ has a surprising connection with the PSH-algebra $Z(\F_q):= \bigoplus_{n\geq 0} \R(\GL_n(\F_q))$, in which the multiplication and comultiplication are given by parabolic induction and restriction. It is known (see \cite[Page 131]{Zelevinsky}) that the number of cuspidal elements of degree $n$ in this algebra is equal to the number of irreducible monic polynomials of degree $n$ in $\F_q[x]$. By taking minimal polynomials, this is also the number of $Gal(\F_{q^n}/\F_q)$-orbits in the set $\{a\in \F_{q^n}| \F_q[a] = \F_{q^n}\}$ of primitive elements in the extension $\F_{q^n}$. Recall that a \emph{normal basis} for a Galois extension $F_2/F_1$ is a basis of the form $\{g(a)|g\in Gal(F_2/F_1)\}$. By the normal basis Theorem (see \cite[Section 4.14]{Jac}), a normal basis always exists. By considering the finite extension $\F_{q^n}/\F_q$ we have the following lemma:
\begin{lemma}\label{lem:normalbasis} There is an element $a\in \F_{q^n}$ such that $\{a,a^q,a^{q^2},\ldots, a^{q^{n-1}}\}$ is a basis for $\F_{q^n}$ over $\F_q$. 
\end{lemma}
Recall that we denote by $Zel$ the universal PSH-algebra with a single cuspidal element. 
Uri Onn observed the following interesting property of the algebra $\Ainf$ in this case:
\begin{proposition}\label{prop:uri} In the case where $k=q=p^l$ we have $$\Ainf\cong Z(\F_q)\ot Zel.$$ 
\end{proposition}
\begin{proof}
It will be enough to show that the two PSH-algebras have the same number of cuspidal elements in every degree. In degree 1 the algebra $Z(\F_q)\ot Zel$ has $q-1+1=q$ cuspidal elements. The same holds for the algebra $\Ainf$, where the cuspidal elements are  $\Tr(T_0),\ldots, \Tr(T_{q-1})$. 

In degree $n>1$, the number of cuspidal elements in the algebra $\Ainf$ is equal to the number of the equivalence classes of sequences $(i_1,\ldots, i_n)$ in $\{0,\ldots, q-1\}^n$ that cannot be written as a proper power of any smaller sequence. This is equivalent to saying that the orbit of $(i_1,\ldots, i_n)$ under the cyclic permutation contains $n$ elements. 

Let $\psi:\{0,1,\ldots, q-1\}\to \F_q$ be a bijection, and let $a\in \F_{q^n}$ be an element that satisfies the condition of Lemma \ref{lem:normalbasis}. Define $\psi(i_1,\ldots, i_n) = \sum_j \psi(i_j)a^{q^{j-1}}$. This gives us a bijection between $\{0,1,\ldots, q-1\}^n$ and $\F_{q^n}$. We claim that $\Tr(T_{i_1}\cdots T_{i_n})$ is a cuspidal element if and only if $\psi(i_1,\ldots,i_n)\in \F_{q^n}$ is a primitive element. Indeed, by considering the action of the generator of $Gal(\F_{q^n}/\F_q)$ we see that 
$$\psi(i_1,\ldots, i_n)^q = \sum_j \psi(i_j)^q a^{q^{j}} = \sum_j \psi(i_{j-1})a^{q^{j-1}} = \psi(i_n,i_1,\ldots, i_{n-1}).$$ 
The $Gal(\F_{q^n}/\F_q)$-orbit of $\psi(i_1,\ldots, i_n)$ then contains $n$ elements if and only if the $c_n$-orbit of $(i_1,\ldots, i_n)$ contains $n$ elements. This shows that there is a bijection between the cuspidal elements in $\Ainf_n$ and the cuspidal elements in $\R(\GL_n(\F_q))$. Since $Zel$ has no cuspidal elements in degree $n>1$, we get the result. 
 \end{proof}
 For every $d\geq 0$ the algebra $\Ainf$ contains an ideal $I_d$ such that $\Ainf/I_d\cong K[x_{ij}^{(l)}]^{\GL_d}$, the algebra of invariants of $k$-tuples of matrices of degree $d$, where the action of $\GL_d$ is given by simultaneous conjugation. These algebras were studied extensively, see \cite{Procesi,Teranishi,Nakamoto,ADS,BD,Hoge,Razmyslov}

 The proposition above provides an isomorphism $\Theta:\Ainf\to Z(\F_q)\ot Zel$. 
 The main difficulty in studying the algebras $K[x_{ij}^{(l)}]^{\GL_d}$ is in studying the relations arising from the ideal $I_d$.  
 This raises the following natural question:
 \begin{question} Do the ideals $\Theta(I_d)$ have a representation theoretical interpretation? 
 \end{question}
 \subsection{The case $m=2$}
For convenience, we will assume we have a structure tensor $x_1$ of type $(2,2)$. We consider the graph \\ 
\begin{center}
\begin{tikzpicture}
	\begin{pgfonlayer}{nodelayer}
		\node [style=graph node] (0) at (7, 4.75) {};
		\node [style=none] (1) at (7, 4.75) {};
		\node [style=none] (2) at (7, 4.75) {};
		\node [style=none] (3) at (4.75, 4.75) {};
		\node [style=none] (4) at (9.25, 4.75) {};
		\node [style=none] (5) at (6, 6.25) {$(1,1)$};
		\node [style=none] (6) at (8.25, 6.25) {$(2,2)$};
		\node [style=none] (7) at (7.5, 4.75) {1};
		\node [style=none] (8) at (3.5, 4.75) {$\Gamma_0:=$};
	\end{pgfonlayer}
	\begin{pgfonlayer}{edgelayer}
		\draw [style=new edge style 0, bend left=90, looseness=1.75] (4.center) to (2.center);
		\draw [style=new edge style 0, bend right=90, looseness=1.75] (3.center) to (2.center);
		\draw [style=new edge style 0, bend left=90, looseness=1.75] (2.center) to (4.center);
		\draw [style=new edge style 0, bend left=270, looseness=1.75] (2.center) to (3.center);
	\end{pgfonlayer}
\end{tikzpicture}
\end{center}
that corresponds to the invariant $\Tr(x_1)$, where we consider $x_1$ as a map from $W\ot W\to W\ot W$. 
The graph $\Ga_0$ is then a cuspidal element of degree 1 in the algebra $\Ainf_{\Ga_0}$. 

We describe now the cuspidal elements of degree 2 in $\Ainf_{\Ga_0}$. Write $\pi_1(\Ga_0,v) = \langle z_1,z_2\rangle$, where $z_i$ is given by the edge with coloring $(i,i)$ for $i=1,2$. The subgroup $F_2^0$ is then freely generated by the elements $\{z_1^az_2^{-a}\}_{a\in \Z}$. Define $F_2^0\to S_2$ by $z_1^az_2^{-a}\mapsto (12)^a$. By considering the action of $z_1$ on the generators of $F_2^0$ it can be shown that this is the only non-trivial homomorphism that is invariant under the conjugation action of $z_1$ on $F_2^0$. 
Thus, if we think of this homomorphism as an orbit $\Ow$ we have $k(\Ow)=1$. Since $S_2$ is abelian it holds that $\Aut_{F_2^0}(\Ow)=S_2$. This group has two one-dimensional representations, denoted by $K_0$ and $K_1$. 
The automorphism $\rho(\Ow)$ is trivial, and therefore $l(K_0)=l(K_1)=1$.  The resulting cuspidal elements in degree 2 are then $\frac{1}{2}(\Ga_1+\Ga_2)$ and $\frac{1}{2}(\Ga_1-\Ga_2)$, where 
\begin{center}\scalebox{0.9}{
\begin{tikzpicture}
	\begin{pgfonlayer}{nodelayer}
		\node [style=none] (8) at (15.75, 7.25) {$\Gamma_1:=$};
		\node [style=graph node] (13) at (19, 7.25) {};
		\node [style=none] (14) at (19, 7.25) {};
		\node [style=none] (15) at (19, 7.25) {};
		\node [style=none] (16) at (16.75, 7.25) {};
		\node [style=none] (17) at (18, 8.75) {$(1,1)$};
		\node [style=graph node] (18) at (22.25, 7.25) {};
		\node [style=none] (19) at (22.25, 7.25) {};
		\node [style=none] (20) at (22.25, 7.25) {};
		\node [style=none] (21) at (20, 7.25) {};
		\node [style=none] (22) at (21.25, 8.75) {$(1,1)$};
		\node [style=none] (23) at (20.75, 9.5) {};
		\node [style=none] (24) at (21, 4.5) {};
		\node [style=none] (25) at (20.5, 10) {$(2,2)$};
		\node [style=none] (26) at (20.5, 4) {$(2,2)$};
		\node [style=none] (27) at (24.5, 7.25) {$\Gamma_2:=$};
		\node [style=graph node] (28) at (27.75, 7.25) {};
		\node [style=none] (29) at (27.75, 7.25) {};
		\node [style=none] (30) at (27.75, 7.25) {};
		\node [style=none] (31) at (25.5, 7.25) {};
		\node [style=none] (32) at (29.5, 10) {$(1,1)$};
		\node [style=graph node] (33) at (31, 7.25) {};
		\node [style=none] (34) at (31, 7.25) {};
		\node [style=none] (35) at (31, 7.25) {};
		\node [style=none] (36) at (28.75, 7.25) {};
		\node [style=none] (37) at (29.5, 4) {$(1,1)$};
		\node [style=none] (38) at (29.5, 9.5) {};
		\node [style=none] (39) at (29.75, 4.5) {};
		\node [style=none] (40) at (29.75, 8.75) {$(2,2)$};
		\node [style=none] (41) at (26.5, 8.75) {$(2,2)$};
		\node [style=none] (42) at (31.5, 7.25) {1};
		\node [style=none] (43) at (27.25, 7.25) {1};
		\node [style=none] (44) at (18.5, 7.25) {1};
		\node [style=none] (45) at (22.75, 7.25) {1};
	\end{pgfonlayer}
	\begin{pgfonlayer}{edgelayer}
		\draw [style=new edge style 0, bend right=90, looseness=1.75] (16.center) to (15.center);
		\draw [style=new edge style 0, bend left=270, looseness=1.75] (15.center) to (16.center);
		\draw [style=new edge style 0, bend right=90, looseness=1.75] (21.center) to (20.center);
		\draw [style=new edge style 0, bend left=270, looseness=1.75] (20.center) to (21.center);
		\draw [style=new edge style 0, in=180, out=90] (15.center) to (23.center);
		\draw [style=new edge style 0, in=90, out=0, looseness=1.25] (23.center) to (20.center);
		\draw [style=new edge style 0, in=0, out=-90] (20.center) to (24.center);
		\draw [style=new edge style 0, in=-90, out=-180] (24.center) to (15.center);
		\draw [style=new edge style 0, bend right=90, looseness=1.75] (31.center) to (30.center);
		\draw [style=new edge style 0, bend left=270, looseness=1.75] (30.center) to (31.center);
		\draw [style=new edge style 0, bend right=90, looseness=1.75] (36.center) to (35.center);
		\draw [style=new edge style 0, bend left=270, looseness=1.75] (35.center) to (36.center);
		\draw [style=new edge style 0, in=180, out=90] (30.center) to (38.center);
		\draw [style=new edge style 0, in=90, out=0, looseness=1.25] (38.center) to (35.center);
		\draw [style=new edge style 0, in=0, out=-90] (35.center) to (39.center);
		\draw [style=new edge style 0, in=-90, out=-180] (39.center) to (30.center);
	\end{pgfonlayer}
\end{tikzpicture}}
\end{center}
 
\section{Finitely generated groups}\label{sec:fggrp}
Let $G$ be a finitely generated group. Assume that $G$ admits a surjective group homomorphism $q:G\to\Z$. We can find a generating set $\{g_1,\ldots, g_m\}$ for $G$ such that $q(g_i)=1$ for $i=1,\ldots, m$. This generating set induces a surjective homomorphism $p:F_m\to G$ given by $p(z_i)=g_i$. The homomorphism $\upsilon$ then factors as $\upsilon = qp$.
 
Write $G^0 = \Ker(q)$, and write $x\in G$ for a preimage of $1\in\Z$. We thus have $G=\langle x\rangle\rtimes G^0$, and $\Ker(p)\subseteq F_m^0$. Let $\Ga_0$ be a graph with $\pi_1(\Ga_0,v)=F_m$. Then $\Ainf_{\Ga_0}$ has a graded subspace $\Ainf_G$ such that $(\Ainf_G)_n$ is spanned by all $\ol{(\sigma_1,\ldots, \sigma_m)}$ such that $\phi_{(\sigma_i)}:F_m\to S_n$ factors through $p$. It is easy to see that $\Ainf_G$ is a graded Hopf subalgebra of $\Ainf_{\Ga_0}$. The following result is immediate using the correspondence between covering spaces and $G$-sets:
\begin{proposition}
Assume that $G=\pi_1(T,t)$ for some topological space $T$ that admits a universal covering space. 
The algebra $\Ainf_G$ is isomorphic to the polynomial algebra on the isomorphism classes of the finite covering spaces of $T$. If $z:\wt{T}\to T$ is an $n$-fold covering space, then $[\wt{T}]$ has degree $n$ in $\Ainf_G$, and $|[\wt{T}]|^2 = |\Aut_T(\wt{T})|$, the cardinality of the automorphism group of $\wt{T}$ as a covering space. 
\end{proposition}
The following definition is a direct generalization of the definitions given in Section \ref{sec:decomposing} in case $G=F_m$. 
\begin{definition}
Let $G,G^0$, and $x$ be as above. A strongly finite transitive $G^0$-set is a finite transitive $G^0$-set $\Ow$ such that $x^k(\Ow)\cong \Ow$ for some $k>0$. We write $k(\Ow)$ for the minimal $k$ that satisfies this condition.
\end{definition}
Choose an isomorphism $\Phi:x^k(\Ow)\to \Ow$. Conjugation by $\Phi$ induces an automorphism $\rho(\Ow):\Aut_{G^0}(\Ow)\to \Aut_{G^0}(\Ow)$. 
\begin{definition} Write 
$FT(G,G^0)$ for the set of all tuples of the form $(\Ow,\ol{[W]})$ such that $\Ow$ is a strongly finite transitive $G^0$-set, $W$ is an irreducible $\Aut_{G^0}(\Ow)$-representation, and $\ol{[W_1]} = \ol{[W_2]}$ if and only if $W_1\cong (\rho(\Ow)^i)^*(W_2)$ for some $i$. We define the degree of $(\Ow,\ol{[W]})$ to be $|\Ow|k(\Ow)l(W)$, where $l(W)$ is the minimal number $l$ for which $(\rho^l)^*(W)\cong W$. We write $FT(G,G^0)_d$ for the subset of elements of $FT(G,G^0)$ of degree $d$. 
\end{definition}
\begin{proposition}\label{prop:fggrp}
The intersection $\H_G:=\H_{\Ga_0}\cap \Ainf_G$ is a PSH-algebra. The cuspidal elements of degree $n$ in this PSH algebras are in one to one correspondence with elements of degree $n$ in $FT(G,G^0)$. As a result, the number of conjugacy classes of index $n$ subgroups in $G$ is equal to $\sum_{d|n}|FT(G,G^0)|_d$.
\end{proposition}
\begin{proof}
The fact that the intersection is a PSH-algebra follows from the fact that $\Ainf_G = \bigotimes_{\Ow,i} \Ainf_{\Ga_0,\Ow,i}$, where we take the tensor product over all orbits $\Ow$ on which $\Ker(p)$ acts trivially. 
For the second statement, we compare the number of generators of the polynomial ring $\Ainf_G$ and $\H_G$. On the one hand, $\Ainf_G$ is generated by finite transitive $G$-sets. On the other hand, the cuspidal elements in $\H_G$ are in one to one correspondence with the set $FT(G,G^0)$. By Lemma \ref{lem:numofgens} we get the result. 
\end{proof}
We give now a few concrete examples of this formula.
\begin{proposition}
Let $G=G^0\times \langle x \rangle$. Then the number of conjugacy classes of index $n$ subgroups in $G$ is equal to the number of irreducible representations of groups of the form $\Aut_{G^0}(\Ow)$, where $\Ow$ is a finite transitive $G^0$-set, and $|\Ow|$ divides $n$. 
\end{proposition}
\begin{proof}
This follows directly from the above proposition, using the fact that the action of $x$ by conjugation is trivial and therefore $k(\Ow) = l(W)=1$ for every finite transitive $G^0$-set $\Ow$ and every irreducible $\Aut_{G^0}(\Ow)$-representation $W$. 
\end{proof}
In case the group $G$ is abelian we get the following recursive formula:
\begin{corollary}
Assume that $G= G^0\times \langle x \rangle$ is abelian. Write $a_n$ for the number of index $n$ subgroups of $G^0$. Then the number of index $n$ subgroups of $G$ is equal to $\sum_{d|n} da_d$. 
\end{corollary}
\begin{proof}
We use the above proposition. If $|\Ow|=d$ and $H\subseteq G^0$ is the stabilizer of an element in $\Ow$, then $\Aut_{G^0}(\Ow)\cong N_{G^0}(H)/H = G^0/H$ is an abelian group of order $d$, and therefore has exactly $d$ irreducible representations.
\end{proof}
\begin{remark}
The referee suggested the following alternative proof for this corollary:
For every index $n$ subgroup $H$ of $G$ we can consider the subgroup $H_0:= H\cap G_0$. It holds that $|G_0/H_0|=d$ is a divisor of $n$. The fact that $H$ has index $n$ in $G$ implies that $H = \langle H_0, yx^{n/d}\rangle$ for some $y\in G_0$. We get a bijection between subgroups $H$ of index $n$ that correspond to a given $H_0$ and $G_0/H_0$ by sending $H$ to $\ol{y}$. 
\end{remark}

This enables us to write down a formula for the number of index $n$ subgroups in any finitely generated abelian group. Let $G$ be such a group. We can write $G=D\times \Z^n$ for some $n$, where $D$ is a finite group. For $d\in\N$, write $a_d$ for the number of index $d$ subgroups in $D$. We claim the following:
\begin{proposition}
The number of index $m$ subgroups of $G$ is equal to $$\sum_{d_1|d_2|\cdots | d_n|m}a_{d_1}d_1d_2\cdots d_n.$$
\end{proposition}
\begin{proof}
Write $a(m,n)$ for the above number. We proceed by induction on $n$. For $n=1$ the result follows directly from the above corollary with $D=G^0$. For $n>1$ the above corollary gives $$a(m,n) = \sum_{d|m} a(d,n-1)d,$$ and a direct verification shows that the above formula satisfies this recursive relation. 
\end{proof}
\begin{remark} Chapter 15 in \cite{LS} has five different proofs for this proposition in case $D=1$. 
\end{remark}

We next apply the formula to some Baumslag-Solitar groups. We mention that the subgroup growth in Baumslag-Solitar group was studied in 
\cite{Ge05} and in \cite{Ke20}. The methods of this paper are relevant to study the growth of the number of conjugacy classes of index $n$ subgroups, while subgroup growth deals with the growth of the number of index $n$ subgroups. 

We begin with the fundamental group of the Klein bottle, $G= \langle a,b| aba^{-1} = b^{-1}\rangle$. We define 
$$a_d := \begin{cases} 1 &\text{ if } d \text{ is odd} \\ 
\frac{d+6}{4} &\text{ if } d = 2 \text{ mod } 4 \\
\frac{d+4}{4} &\text{ if }  d=0 \text{ mod } 4\end{cases}.$$
We claim the following:
\begin{proposition}
The number of conjugacy classes of index $n$ subgroups in $G$ is $\sum_{d|n} a_d$. 
\end{proposition}
\begin{proof}
We can write $G=G^0\rtimes \langle a \rangle$ where $G^0 = \langle b\rangle$.
To prove the result, it will be enough to show that the algebra $\H_G$ has exactly $a_d$ cuspidal elements in degree $d$ for every $d\in\N$. 
For every $d\in \N$ there is exactly one transitive $G^0$-set of cardinality $d$, $\Ow_d:=\langle b\rangle / \langle b^d\rangle$. 
It holds that $\Aut_{G^0}(\Ow_d)\cong \Z/d$. The action of $\rho(\Ow)$ is then just given by inversion. We have $k(\Ow_d)=1$. 

If $d$ is odd, then the action of $\rho(\Ow_d)$ on $\Irr(\Aut_G^0(\Ow_d))$ has $\frac{d-1}{2}$ orbits with 2 elements, and one orbit with one element (of the trivial representation of $\Z/d$. This gives $\frac{d-1}{2}$ cuspidal elements in degree $2d$ and 1 cuspidal element in degree $d$. 
If $d$ is even, then the action of $\rho(\Ow_d)$ on $\Irr(\Aut_G^0(\Ow_d))$ has $\frac{d-2}{2}$ orbits with two elements and two orbits with one element (where a generator of $\Z/d$ acts by 1 or -1). This gives $\frac{d-2}{2}$ cuspidal elements in degree 2d and 2 cuspidal elements in degree $d$.
 
So if $d$ is odd we get just one cuspidal element in degree $d$, if $d$ is of the form $2m$ where $m$ is odd we get $2+\frac{m-1}{2} = \frac{d+6}{4}$ cuspidal elements of degree $d$, and if $d=4m$ then we get $2+\frac{2m-2}{2} = \frac{d+4}{4}$ cuspidal elements of degree $d$. This gives us exactly the formula we have for $a_d$ above.
\end{proof}
Let now $d>1$ be a natural number. We consider the group 
$G=\langle a,b | aba^{-1}=b^d\rangle$. The group $G$ can be written as a semidirect product $G=G^0\rtimes \langle a\rangle$, where $G^0 = \langle a^iba^{-i}\rangle_{i\in\Z}$. Write $S= \{(m,i) | m,i\in \N, gcd(m,d)=1, 0\leq i< m\}$. Define an equivalence relation $(m_1,i_1)\sim (m_2,i_2)$ if and only if $m_1=m_2$ and there is $j\in\Z$ such that $i_1 = i_2d^j$ mod $m_1$. Define the degree of $\ol{(m,i)}\in S/\sim$ to be $m\cdot ord_i(d)$, where $ord_i(d)$ stands for the order of $d$ in $(\Z/m')^{\times}$, with $m' = m/gcd(m,i)$. We claim the following:
\begin{proposition}
The number of conjugacy classes of index $n$ subgroups in $G$ is equal to the number of elements $\ol{(m,i)}\in S/\sim$ with $deg(\ol{(m,i)})|n$.
\end{proposition}
\begin{proof}
As before, it will be enough to show correspondence between the elements of $S/\sim$ and the set $FT(G,G^0)$ that preserves degrees. The group $G^0$ is abelian. Let $\Ow$ be a strongly finite transitive $G^0$-set, and let $H$ be the stabilizer of an element in $\Ow$. Let $k=k(\Ow)>0$ be the minimal number for which $a^k(\Ow)\cong \Ow$. This implies that $a^kHa^{-k}$ is conjugate in $G^0$ to $H$. Since $G^0$ is abelian, it holds that $a^kHa^{-k}=H$.
Since conjugation by $a$ acts on $G^0$, and in particular on $H$, by raising to the $d$-th power, it holds that $H\subseteq a H a^{-1}$. It then holds that $H\subseteq aHa^{-1}\subseteq a^2Ha^{-2}\subseteq\cdots\subseteq a^kHa^{-k} = H$, so all of these subgroups are equal. This implies, in particular, that $H=aHa^{-1}$, and so $k=1$. 

The group $\Aut_{G^0}(\Ow)$ is isomorphic with $N_{G^0}(H)/H = G^0/H$ since $G^0$ is abelian. Moreover, conjugation by $a$ acts on this group by sending $x$ to $x^{d}$. This implies that when we consider the finite abelian group $G^0/H$ as a $\Z$-module, multiplication by $d$ becomes an invertible operation. In particular, this means that $gcd(|G^0/H|,d)=1$. Write $|G^0/H|=m$. 

Since $G^0 = \bigcup_{i\geq 0} a^{-i}\langle b\rangle a^i$ is a union of a chain of cyclic groups, the same is true for the quotient $G^0/H$. Since $G^0/H$ is also finite, it must be cyclic. Write $e$ for the inverse of $d$ in $(\Z/m)^{\times}$. We get an isomorphism $G^0/H\cong \Z/m$ by sending $a^{-i}ba^i$ to $e^i$. In particular, we see that $H = (G^0)^m=\{g^m|g\in G^0\}$, and so $H$ is the unique subgroup of $G^0$ of index $m$.
 
We thus see that for every $m$ such that $gcd(m,d)=1$ there is a unique strongly finite transitive $G^0$-set of cardinality $m$. The automorphism group of this orbit is $\Z/m$, and the action of $a$ by conjugation is given by multiplication by $d$. The dual action, on $\Irr(\Z/m)\cong \Z/m$ is also multiplication by $d$. So the elements in $FT(G,G^0)$ correspond to pairs $(m,i)$ such that $gcd(m,d)=1$, and such that $0\leq i<m$. Two such pairs $(m_1,i_1)$ and $(m_2,i_2)$ are equivalent if and only if $m_1=m_2$, and $i_1=i_2d^j$ mod $m$, for some $j\in \Z$. The degree of the element $(\Ow,\ol{[W]})$ that corresponds to $(m,i)$ is then $mk(\Ow)l(\Ow)$. We have already seen that $k(\Ow)=1$. The number $l(\Ow)$ is the cardinality of the orbit of $[W]\in \Irr(\Z/m)\cong \Z/m$ under the action of multiplication by $d$. In other words, it is the minimal number $l$ such that $$d^li = i \text{ mod } m.$$ The last equation is equivalent to $$d^l = 1\text{ mod }m/gcd(m,i).$$ In other words $l$ is the order of $d$ in the group $(\Z/(m/gcd(m,i)))^{\times}$ as claimed.
\end{proof}

\section{Hilbert series}\label{sec:hilbert}
We consider now the Hilbert series of the algebra $\Ainf_{\Ga_0}$, where we assume that $\pi_1(\Ga_0,v)=F_m$. Quotients of this algebra were studied in the paper \cite{meir3} in case $m=2$. Recall from Section \ref{sec:prelim} that for partitions $\mu,\la_i\parti n$ we write 
$$g(\la_1,\ldots,\la_m,\mu) = \dim\Hom_{S_n}(\S_{\la_1}\ot\cdots\ot \S_{\la_m},\S_{\mu}).$$
Write $T(F_m)$ for the set of isomorphism classes of transitive $F_m$-sets. 
We claim the following:
\begin{theorem} The Hilbert series of $\Ainf_{\Ga_0}$ is equal to the following series: 
$$\sum_{n\geq 0}\dim((\Ainf_{\Ga_0})_n)X^n = \sum_n\sum_{\mu,\la_1,\ldots, \la_m\parti n}g(\la_1,\ldots,\la_m,\mu)^2X^n =$$ $$ \prod_{\Ow\in T(F_m)}\frac{1}{1-X^{|\Ow|}} = \prod_{(\Ow,\ol{[W]})\in FT(F_m,F_m^0),n}\frac{1}{1-X^{|\Ow|k(\Ow)l([W])n}}$$
\end{theorem}
\begin{proof}
For the first series, we use the fact that by Proposition \ref{prop:splitting1} it holds that $(\Ainf_{\Ga_0})_n\cong (KS_n^m)_n$. 
We will show that $\dim(KS_n^m)_{S_n} = \sum_{\la_i,\mu\parti n } g(\la_1,\ldots, \la_m,\mu)^2$. Indeed, taking the $m$-fold tensor product of the Wedderburn decomposition $$KS_n = \bigoplus_{\la\parti n}\End(\S_{\la})$$ gives $$KS_n^m = \bigoplus_{\la_i\parti n} \End(\S_{\la_1}\boxtimes\cdots\boxtimes \S_{\la_m}).$$
As $S_n$-representations, we have $$\S_{\la_1}\ot\cdots\ot \S_{\la_m} = \bigoplus_{\mu\parti n}\S_{\mu}^{\oplus g(\la_1,\ldots, \la_m,\mu)}.$$
So $$(KS_n^m)_{S_n}\cong (KS_n^m)^{S_n} =  \bigoplus_{\la_i,\mu\parti n} \End_{S_n}(\bigoplus_{\mu\parti n}\S_{\mu}^{\oplus g(\la_1,\ldots, \la_m,\mu)}).$$ By taking the dimensions of both spaces we get the result.

The second series follows from the fact that $\Ainf_{\Ga_0}$ is a polynomial algebra, where a set of variables is given by all finite connected coverings of $\Ga_0$, and such coverings are in one to one correspondence with finite transitive $F_m$-sets. 

For the third series, we use the PSH-algebra structure, and the fact that a cuspidal element of degree $k$ gives rise to infinitely many variables $x_1,x_2,\ldots$ with degrees $deg(x_n)=nk$. 
\end{proof}

\section{Concluding remarks}\label{sec:remarks}
The PSH-algebra $\H_{\Ga_0}$ gives a different generating set for the algebra $\Ainf_{\Ga_0}$. Since $\Ainf\cong \bigotimes_{\Ga_0}\Ainf_{\Ga_0}$ this also gives us a different generating set for $\Ainf$. In the case of a single linear endomorphism $T$ we have seen in \cite[Section 10]{meirUIR} that this generating set was very useful in describing the ideals $I_d$. Indeed, in this case we have $\Ainf = \Ainf_{\Ga_0}$ where $\Ga_0$ is the graph that corresponds to $\Tr(T)$, $\H_{\Ga_0} = \Z[Y_1,Y_2,\ldots]$ is a polynomial ring in infinitely many variables, and $I_d\cap \H_{\Ga_0} = (Y_{d+1},Y_{d+2},\ldots)$, which gives a very neat description of $\Ainf/I_d$. In Section \ref{sec:examples} we have constructed an isomorphism $\Phi:\Ainf\cong Zel\ot \Z(\F_q)$ for the type $((1,1)^q)$, where $q$ is a prime power. It is not clear if the ideals $\Phi(I_d)$ have an interpretation in terms of the representation theory of $\GL_n(\F_q)$. 
\begin{question} Do the ideals $I_d$ admit a generating set that is easy to describe in terms of the generators of $\H_{\Ga_0}$, for different $\Ga_0$?
\end{question}
We can also localize the above question.
\begin{question} For a given irreducible graph $\Ga_0$, do the ideals $I_d\cap \H_{\Ga_0}$ admit a generating set that is easy to describe in terms of the generators of $\H_{\Ga_0}$?
\end{question}

There is another important aspect that arises in the case of a single endomorphism. The elements of the PSH-algebra are $\Q$-linear combinations of the diagram invariants, but they still make sense over $\Z$ and therefore over any field. For example, the invariant $c_2$ can be written as 
\begin{center}
\begin{tikzpicture}
	\begin{pgfonlayer}{nodelayer}
		\node [style=1function] (0) at (-10, 4.75) {$T$};
		\node [style=1function] (1) at (-9, 4.75) {$T$};
		\node [style=1function] (2) at (-5.75, 4.75) {$T$};
		\node [style=1function] (3) at (-3.5, 4.75) {$T$};
		\node [style=none] (4) at (-10, 5.75) {};
		\node [style=none] (5) at (-10, 3.75) {};
		\node [style=none] (6) at (-11.25, 3.75) {};
		\node [style=none] (7) at (-11.25, 5.75) {};
		\node [style=none] (8) at (-9, 5.75) {};
		\node [style=none] (9) at (-9, 3.75) {};
		\node [style=none] (10) at (-7.75, 3.75) {};
		\node [style=none] (11) at (-7.75, 5.75) {};
		\node [style=none] (12) at (-4.75, 5.75) {};
		\node [style=none] (13) at (-5.75, 5.75) {};
		\node [style=none] (14) at (-4.75, 3.75) {};
		\node [style=none] (15) at (-3.5, 3.75) {};
		\node [style=none] (16) at (-3.5, 5.75) {};
		\node [style=none] (17) at (-2.5, 5.75) {};
		\node [style=none] (18) at (-2.5, 3.75) {};
		\node [style=none] (19) at (-5.75, 3.75) {};
		\node [style=none] (20) at (-7, 4.75) {$-\frac{1}{2}$};
		\node [style=none] (21) at (-12, 4.75) {$\frac{1}{2}$};
	\end{pgfonlayer}
	\begin{pgfonlayer}{edgelayer}
		\draw [bend left=270, looseness=1.75] (4.center) to (7.center);
		\draw (7.center) to (6.center);
		\draw [bend right=90, looseness=1.75] (6.center) to (5.center);
		\draw (5.center) to (0);
		\draw (4.center) to (0);
		\draw (9.center) to (1);
		\draw [bend right=90, looseness=1.75] (9.center) to (10.center);
		\draw (10.center) to (11.center);
		\draw [bend left=270, looseness=1.75] (11.center) to (8.center);
		\draw (8.center) to (1);
		\draw (2) to (13.center);
		\draw [bend right=90, looseness=1.25] (14.center) to (15.center);
		\draw (15.center) to (3);
		\draw (3) to (16.center);
		\draw [bend left=90, looseness=2.00] (16.center) to (17.center);
		\draw (17.center) to (18.center);
		\draw [bend left=90, looseness=0.75] (18.center) to (19.center);
		\draw (19.center) to (2);
		\draw [bend left=90, looseness=2.25] (13.center) to (12.center);
		\draw (12.center) to (14.center);
	\end{pgfonlayer}
\end{tikzpicture},
\end{center}
but it gives a $\Z$-linear polynomials in the entries of the matrix. 
This raises the following question: given a type $((p_i,q_i))$, we can consider the structure constants for a module of rank $d$ as a scheme $U_d$ over $\Z$. More precisely, we can define $$U_d = \text{Spec}(\Z[(x_i)_{a_{j_1},\ldots, a_{j_{q_i}}}^{b_{k_1},\ldots, b_{k_{p_i}}}]_{1\leq i\leq r, 1\leq j_t,k_s\leq d}).$$ 
This scheme is equipped with an action of the group scheme $\GL_d$, and we can consider the ring of invariants $\Z[U_d]^{\GL_d}$. 
\begin{question} Are the PSH-algebras constructed in this paper naturally contained in $\Z[U_d]^{\GL_d}$? If so, do they coincide with $\Z[U_d]^{\GL_d}$?
\end{question}
An answer to this question will enable us to tackle questions in invariant theory over the integers and also over fields of positive characteristic. 
\section*{Acknowledgements}
I would like to thank Alex Lubotzky for discussions about questions regarding subgroups growth, Uri Onn for Proposition \ref{prop:uri}, and the referee for their careful reading of the manuscript and their suggestions, which helped me to improve it.  

\section*{declarations}
\subsection{Ethical approval} 
NOT APPLICABLE
\subsection{Funding}
NOT APPLICABLE
\subsection{Availability of data and materials}
NOT APPLICABLE

\end{document}